\documentclass{article}
\usepackage{amsmath}
\usepackage{amsfonts}
\usepackage{amsthm}
\usepackage{amssymb, setspace}
\usepackage{mathtools}
\usepackage{lipsum}
\usepackage{color,graphicx,epsfig,geometry,fancyhdr,hyperref}
\usepackage[T1]{fontenc}
\usepackage{float}
\usepackage{subfigure} 
\usepackage{commath}
\usepackage{bbm}
\usepackage{amsmath}
\usepackage{amsfonts}
\usepackage{amssymb} 
\usepackage{amsthm,color}
\usepackage{mathrsfs}
\usepackage{url,algorithm,algorithmic}
\usepackage{amsmath,amsthm,amssymb}
\usepackage{subfigure} 
\usepackage{caption} 
\usepackage{tabularx} 
\usepackage{graphicx,mathtools} 
\usepackage{booktabs} 
\usepackage{cite}  
\numberwithin{equation}{section}
\usepackage{mathrsfs,fleqn}
\usepackage{pgfplots}
\pgfplotsset{compat=newest}
\newtheorem{theorem}{Theorem}[section]

\newtheorem{lemma}{Lemma}[section]

\usepackage{comment}
\usepackage{ulem}

\begin{document}

\begin{flushleft}
\Large 
\noindent{\bf \Large Sampling methods for the inverse cavity scattering problem of biharmonic waves}

\end{flushleft}

\vspace{0.2in}

{\bf  \large Isaac Harris }\\
\indent {\small Department of Mathematics, Purdue University, West Lafayette, IN 47907, USA, }\\
\indent {\small Email: \texttt{harri814@purdue.edu} }\\

{\bf  \large Peijun Li}\\
\indent {\small SKLMS, ICMSEC, Academy of Mathematics and Systems Science, } \\
\indent {\small Chinese Academy of Sciences, Beijing 100190, China,}\\
\indent {\small Email: \texttt{lipeijun@lsec.cc.ac.cn}}\\

{\bf  \large General Ozochiawaeze}\\
\indent {\small Department of Mathematics, Purdue University, West Lafayette, Indiana 47907, USA, } \\
\indent {\small Email: \texttt{oozochia@purdue.edu}}\\

\vspace{0.2in}

\begin{abstract}
\noindent This paper addresses the inverse problem of qualitatively recovering a clamped cavity in a thin elastic plate using far-field measurements. We present a strengthened analysis of the linear sampling method by carefully examining the range of the far-field operator and employing the reciprocity relation of the biharmonic far-field pattern. In addition, we implement both the linear sampling method for reconstructing the cavity and the extended sampling method for localizing the cavity under limited-aperture data. Numerical experiments demonstrate the effectiveness and robustness of both methods.
\end{abstract}
%
%

\section{Introduction}
This paper investigates the inverse cavity scattering problem in a thin, infinitely extended elastic plate, where the objective is to determine an unknown cavity from the scattering of time-harmonic flexural waves governed by the biharmonic wave equation. Scattering of biharmonic waves in thin plates has diverse applications, including the design of ultra-broadband elastic cloaks for vibration control in vehicles and earthquake resistant buildings \cite{farhat2014platonic, PhysRevB.79.033102, PhysRevB.85.020301, PhysRevLett.103.024301, sozio2023optimal}. Platonic crystals, which are periodic arrays of cavities, enable wave manipulation similar to photonic and phononic crystals  \cite{gao2018theoretical,poulton2010convergence}. Acoustic black holes passively trap flexural waves, supporting applications in noise reduction, energy harvesting, and biomedical devices \cite{PELAT2020115316, krylov2014acoustic}. The biharmonic model also underlies non-destructive testing and structural health monitoring in aerospace engineering \cite{bourgeois2020linear, MEMMOLO2018568}. Despite these diverse applications, the theoretical study of biharmonic wave scattering remains comparatively underdeveloped. In contrast to classical scattering problems involving wave propagation in unbounded media, biharmonic waves describe out-of-plane displacements in thin elastic plates, presenting unique mathematical and computational challenges.

Computational approaches to inverse scattering are generally classified into optimization-based methods and sampling methods. Optimization-based techniques, though often highly accurate, require good initial guesses and involve significant computational cost due to the repeated solution of direct problems. In contrast, this work focuses on sampling methods, which aim to reconstruct the scatterer's boundary using indicator functions. These methods offer advantages such as reduced dependence on a priori geometric or physical information. However, they typically rely on multistatic measurement data and tend to yield only partial reconstructions. Notable examples include the linear sampling method, factorization method, reverse time migration, enclosure method, probe method, and direct sampling method \cite{cakoni2003mathematical, CCH-IPA, CCH-IPE, kirsch2007factorization, Ikehata2001, IkehataMasaru2022Rtpa, harris2021direct, AyalaRafaelCeja2024Ipas}. We refer to the monograph for a comprehensive overview of qualitative approaches in inverse scattering theory. \cite{alma99169166888701081}

Scattering problems for biharmonic waves have recently received increasing attention. A boundary integral formulation for biharmonic wave scattering was developed in \cite{DongHeping2024ANBI}, and the well-posedness of the associated direct problem was established in \cite{bourgeois2020well}. The uniqueness of the inverse cavity scattering problem was demonstrated in \cite{Dong2023UniquenessOA}, while an optimization-based  approach was proposed in \cite{chang2023optimization}. The application of the linear sampling method (LSM) to the inverse cavity problem in a Kirchhoff--Love plate model with near-field data was investigated in \cite{bourgeois2020linear}. In this work, we adapt the LSM and utilize far-field measurements, which offer several technical advantages for the biharmonic cavity scattering problem. Specifically, the decomposition of the scattered field into Helmholtz and modified Helmholtz components allows to exploit the exponential decay of the modified Helmholtz wave component. Furthermore, far-field measurements only involve plane waves, whereas near-field formulations must account for both the distance-dependent fundamental solution and dipole sources. As a result, far-field data reduce the complexity and quantity of required measurements, making them particularly attractive for practical implementation. Importantly, there exists a one-to-one correspondence between the far-field patterns of biharmonic and Helmholtz scattered waves, which provides a rigorous justification for applying the LSM in this setting.

Recently, the LSM was applied to identify unknown clamped cavities in a Kirchhoff--Love plate model using far-field measurements \cite{guoetal2024}. In that work, the theoretical justification required that the wavenumber not coincide with any Dirichlet eigenvalue of the scatterer. However, our analysis shows that this condition is not essential for accurate reconstruction. By deriving a reciprocity relation for the biharmonic far-field pattern, we relax the restriction on the wavenumber. Moreover, we propose a new factorization of the far-field operator that avoids the need for this additional assumption. In addition, we adapt the extended sampling method (ESM) introduced in \cite{Liu_2018}, enabling accurate localization of the scatterer using far-field measurements from as few as a single incident direction. This enhancement is particularly valuable in engineering applications such as non-destructive testing and structural health monitoring, where acquiring multistatic data is often impractical due to the cost and complexity of deploying multiple sensors. By eliminating the dependence on multistatic measurements, our approach simplifies experimental implementation and reduces resource requirements, thereby improving its feasibility and scalability in real-world settings.

The paper is outlined as follows. In Section \ref{section 2}, we present the problem formulation of the biharmonic wave scattering. Section \ref{section 3} establishes the reciprocity principle for the biharmonic far-field pattern. Section \ref{section 4} introduces a new factorization of the far-field operator and employs the reciprocity principle to justify the LSM for reconstructing unknown cavities using far-field measurements. In Section \ref{section_ESM}, we extend the ESM to recover clamped cavities from limited far-field data. Finally, Section \ref{section 5} provides numerical examples that demonstrate the effectiveness of the LSM with full-aperture far-field data and the ESM with limited-aperture measurements.

\section{Problem formulation}\label{section 2}

Let $D\subset\mathbb R^2$ be a bounded open set with a smooth boundary $\Gamma$, representing the clamped cavity embedded in an infinitely extended, thin, two-dimensional elastic plate governed by the Kirchhoff--Love model in the pure bending regime. Assume that the exterior domain $\mathbb R^2\setminus\overline{D}$ is simply connected. We consider an incident field in the form of a time-harmonic plane wave described by
\begin{align*}
    u^i(x)=\text{e}^{{\rm i}\kappa x\cdot d}, \quad x\in \mathbb R^2,
\end{align*}
where $\kappa>0$ represents the wavenumber, and $d\in\mathbb S^1$ is a unit vector indicating the direction of wave incidence. The total displacement field $u$ is governed by the time-harmonic biharmonic equation
\begin{align}\label{be}
    \Delta^2u-\kappa^4u=0\quad \text{in } \mathbb R^2\setminus\overline{D}.
\end{align}
For simplicity, we impose clamped boundary conditions on the cavity boundary on $\Gamma$, i.e., 
\begin{align}\label{cbc}
    u=0,\quad \partial_{\nu}u=0,
\end{align}
where $\nu$ denotes the unit outward normal to $\Gamma$. 

The total field $u$ is represented as the superposition of the incident and scattered waves:
\begin{align*}
u=u^i+u^s,
\end{align*}
where $u^s$ denotes the scattered field. The scattered field $u^s$ satisfies the following boundary value problem in 
the unbounded exterior domain $\mathbb R^2\setminus\overline{D}$:
\begin{align}
\begin{dcases}\label{1.3}
        \Delta^2 u^s-\kappa^4u^s=0\quad \text{in }\mathbb R^2\setminus\overline{D},\\ 
        u^s=-u^i,\quad \partial_{\nu}u^s=-\partial_{\nu}u^i\quad\text{on }\Gamma,\\
        \lim_{r\to\infty}\sqrt{r}(\partial_ru^s-{\rm i}\kappa u^s)=0, \quad r=|x|, 
\end{dcases}
\end{align}
where the last equation is known as the Sommerfeld radiation condition, ensuring that the scattered wave is outgoing at infinity.

Following \cite{DongHeping2024ANBI}, we perform a biharmonic wave decomposition by introducing two auxiliary functions
\begin{align}\label{1.5}
u_{\text{H}}^s=-\frac{1}{2\kappa^2} (\Delta u^s-\kappa^2u^s),\quad u_{\text{M}}^s=\frac{1}{2\kappa^2}(\Delta u^s+\kappa^2 u^s), 
\end{align}
where $u_{\text{H}}^s$ and $u_{\text{M}}^s$ are referred to as the Helmholtz (or propagating) and modified Helmholtz (or evanescent) components, respectively. Combining \eqref{be} with \eqref{1.5}, we obtain in $\mathbb R^2\setminus\overline{D}$
\begin{align}\label{1.7}
\Delta u_{\text{H}}^s+\kappa^2u_{\text{H}}^s =0, \quad \Delta u_{\text{M}}^s-\kappa^2u_{\text{M}}^s=0.
\end{align}
Moreover, a simple calculation yields 
\begin{align}\label{1.6}
    u^s=u_{\text{H}}^s+u_{\text{M}}^s, \quad \Delta u^s=\kappa^2(u_{\text{M}}^s-u_{\text{H}}^s).
\end{align}
It follows from the clamped boundary conditions \eqref{cbc} that $u_{\text{H}}^s$ and $u_{\text{M}}^s$ satisfy the following coupled boundary conditions on $\Gamma$:
\begin{align*}
    u_{\text{H}}^s+u_{\text{M}}^s =-u^i,\quad \partial_{\nu}u_{\text{H}}^s+\partial_{\nu}u_{\text{M}}^s=-\partial_{\nu}u^i.
\end{align*}
As the propagating wave component, the Helmholtz component $u_{\text{H}}^s$ satisfies the Sommerfeld radiation condition
\begin{align}\label{1.9}
\lim_{r\to\infty}\sqrt{r}(\partial_ru_{\text{H}}^s-{\rm i}\kappa u_{\text{H}}^s)=0.
\end{align}
Since $u_{\text{M}}^s$ is the evanescent wave component, both $u_{\text{M}}^s$ and its radial derivative $\partial_ru_{\text{M}}^s$ decay exponentially as $r\to\infty$. 

Based on the operator splitting, the scattering problem defined in (\ref{1.3}) is equivalent to a coupled boundary value problem for the Helmholtz and modified Helmholtz equations, as given in (\ref{1.7})--(\ref{1.9}). The well-posedness of this coupled boundary value problem, and hence of the original problem (\ref{1.3}), for any wavenumber $\kappa$, is discussed in \cite{bourgeois2020well, YueJunhong2024Nsot}. According to \cite[Proposition 2.2]{bourgeois2020well}, the asymptotic behaviors of $u_{\text{H}}^s$, $u_{\text{M}}^s$, and $\partial_ru_{\text{M}}^s$ are given by
\begin{align*}
|u_{\text{H}}^s|=\mathcal{O}\left(\frac{1}{\sqrt{r}} \right),\quad |u_{\text{M}}^s| = \mathcal{O}\left(\frac{\mathrm{e}^{-\kappa r}}{\sqrt{r}}\right), \quad \text{and } \quad |\partial_ru_{\text{M}}^s|=\mathcal{O}\left(\frac{\mathrm{e}^{-\kappa r}}{\sqrt{r}}\right)
\end{align*}
as $r\to\infty$. 

By Green's representation theorem, the Helmholtz and modified Helmholtz components admit the following boundary integral representations for $x\in \mathbb R^2\setminus\overline{D}$:
\begin{align*}
  u^s_{\rm p}(x)=\int_{\Gamma}\left(u^s_{\text{p}}(y)\partial_{\nu_y}\Phi_{\rm p}(x,y; \kappa)-\partial_{\nu_y}u^s_{\text{p}}(y)\Phi_{\rm p}(x,y; \kappa)\right)\text{d}s(y), 
\end{align*}
where ${\rm p}={\rm H}$ or ${\rm M}$, and $\Phi_{\rm p}(x,y; \kappa)$ denotes the fundamental solution of the corresponding Helmholtz or modified Helmholtz equation, expressed as
\begin{align*}
 \Phi_{\rm p}(x,y; \kappa)=\begin{cases}
                    \dfrac{\rm i}{4}H_0^{(1)}(\kappa|x-y|),\quad {\rm p}={\rm H},\\[5pt]
                    \dfrac{\rm i}{4}H_0^{(1)}({\rm i}\kappa|x-y|),\quad {\rm p}={\rm M}.
                   \end{cases}
\end{align*}
Here, $H^{(1)}_0$ refers to the zeroth-order Hankel function of the first kind.

Given that $u^s$ is radiative, it can be expressed asymptotically as:
\begin{align}{\label{1.10}}
u^s(x)=\frac{\text{e}^{{\rm i}\pi/4}}{\sqrt{8\pi\kappa}}\frac{\text{e}^{{\rm i}\kappa r}}{\sqrt{r}}u^{\infty}(\Hat{x})+O \left(\frac{1}{r^{3/2}}\right),\quad r \to\infty,
\end{align}
where $\Hat{x}\coloneqq x/r\in \mathbb S^1$. The far-field pattern of the scattered field $u^s$ is described by the analytic function $u^{\infty}$, which is defined on the unit circle $\mathbb S^1$. It characterizes the scattering amplitude of the radiating solution $u^s$ in the far-field region. The corresponding inverse problem consists of determining the cavity $D$ from knowledge of the far-field pattern $u^{\infty}$.

Let $\mathcal{F}: L^2(\mathbb S^1)\to L^2(\mathbb S^1)$ denote the far-field operator, specified by:
\begin{align*}
    (\mathcal{F}g)(\Hat{x})\coloneqq \int_{\mathbb S^1}u^{\infty}(\Hat{x},d)g(d) \text{d}s(d). 
\end{align*}
This operator maps a density function $g$, defined on the unit circle of incident directions, to the corresponding superposition of far-field patterns $u^{\infty}(\hat x, d)$ observed in directions $\hat x\in \mathbb S^1$. The inverse scattering problem for a clamped cavity can be stated as follows: Given the far-field operator $\mathcal{F}$ for a range of wavenumbers $\kappa$, determine qualitative characteristics of the clamped cavity $D$ embedded in a thin elastic plate. The uniqueness of this problem has been established in \cite{Dong2023UniquenessOA}.

It is important to note that the direct scattering problem described by (\ref{1.3}) is equivalent to the coupled boundary value problem (\ref{1.7})--(\ref{1.9}). Formulating the problem in this coupled form offers a significant advantage: it transforms a fourth-order boundary value problem into a more tractable system of second-order equations. Moreover, both 
$u^s_{\text{M}}$ and $\partial_r u^s_{\text{M}}$ decay exponentially as $r\to\infty$ for a fixed wavenumber $\kappa$. Consequently, the far-field patterns of $u^s$ and $u^s_{\text{H}}$ are identical.

\section{Reciprocity principle of far-field patterns}\label{section 3}

We consider the recovery of a clamped cavity $D$ from far-field measurements at a fixed wavenumber. The reconstruction relies critically on an approximate solvability condition for the so-called far-field equation, which is derived from a reciprocity principle satisfied by the far-field patterns of the scattered field associated with the biharmonic wave equation. This section addresses the reciprocity relation in detail.

By applying the definition of the far-field pattern in (\ref{1.10}) and Green's second identity, it is shown in \cite{Dong2023UniquenessOA} that the far-field pattern admits the following Green's representation formula:
\begin{align}\label{Green's_far_field}
u^{\infty}(\Hat{x})&= \frac{1}{2}\frac{\text{e}^{\mathrm{i}\pi/4}}{\sqrt{8\pi \kappa}}\int_{\Gamma}\left(u^s(y)\frac{\partial \text{e}^{-\mathrm{i}\kappa \Hat{x}\cdot y}}{\partial \nu(y)}-\text{e}^{-\mathrm{i}\kappa \Hat{x}\cdot y}\frac{\partial u^s}{\partial\nu}(y)\right) \text{d}s(y) \\
&\quad \quad -\frac{1}{2\kappa^2}\frac{\text{e}^{\mathrm{i}\pi/4}}{\sqrt{8\pi \kappa}}\int_{\Gamma}\left(\Delta u^s(y)\frac{\partial \text{e}^{-\mathrm{i}\kappa \Hat{x}\cdot y}}{\partial \nu(y)}-\text{e}^{-\mathrm{i}\kappa \Hat{x}\cdot y}\frac{\partial \Delta u^s}{\partial\nu}(y)\right) \text{d}s(y),\quad \Hat{x}\in \mathbb S^1 \nonumber.
\end{align}

Let $\Omega$ be a bounded domain of class $C^2$. Consider the Hilbert space
\begin{align*}
    H^2(\Omega,\Delta^2) = \left\{w\in H^2(\Omega): \Delta^2 w\in L^2(\Omega) \right\},
\end{align*}
equipped with the norm
\begin{align*}
\|w\|_{H^2(\Omega,\Delta^2)}^2 =\|w\|_{H^2(\Omega)}^2+\|\Delta^2 w\|_{L^2(\Omega)}^2.
\end{align*}
The space $H^2(\Omega,\Delta^2)$ is the maximal domain of the biharmonic operator $\Delta^2$, regarded as an unbounded operator on $L^2(\Omega)$.

According to Green's second identity, for $u,v\in H^2(D,\Delta^2)$, we have
\begin{align}\label{Green's_id}
    \int_D \left\{(\Delta^2u)v-\Delta u\Delta v \right\}\, \text{d}x=\int_{\Gamma}\left\{ v\frac{\partial\Delta u}{\partial\nu}-\Delta u\frac{\partial v}{\partial\nu} \right\} \text{d}s.
\end{align}
Applying identity \eqref{Green's_id} twice and interchanging the roles of $u$ and $v$, we obtain
\begin{align}\label{Green's_2nd_id}
    \int_{D}\left\{ (\Delta^2u)v-(\Delta^2v)u \right\}\text{d}x=\int_{\Gamma}\left\{v\frac{\partial \Delta u}{\partial\nu}-\Delta u\frac{\partial v}{\partial\nu}+\Delta v\frac{\partial u}{\partial\nu}-u\frac{\partial\Delta v}{\partial\nu} \right\} \text{d}s.
\end{align}

The following result establishes the reciprocity principle in two dimensions for the far-field pattern of the radiating solution to the biharmonic wave equation.

\begin{theorem}\label{Reciprocity Thm}
For the direct scattering problem \eqref{1.3} with plane wave incidence $u^i=\mathrm{e}^{\mathrm{i}\kappa x\cdot d}$, 
the far-field pattern $u^{\infty}(\Hat{x},d)$ of the corresponding radiating solution satisfies the identity
    \begin{align*}
        u^{\infty}(-\Hat{x},d)=u^{\infty}(-d,\Hat{x}),\quad\forall\, \Hat{x},d\in \mathbb S^1. 
    \end{align*}
\end{theorem}

\begin{proof}
We begin by observing that applying \eqref{Green's_2nd_id} to the incident fields $u^i(\cdot,\Hat{x})$ and $u^i(\cdot,d)$ inside the scatterer $D$ yields
\begin{align}\label{plane_wave_rep}
&\int_{\Gamma}\left(u^i(\cdot,\Hat{x})\frac{\partial \Delta u^i(\cdot,d)}{\partial\nu}-\Delta u^i(\cdot,d)\frac{\partial u^i(\cdot,\Hat{x})}{\partial\nu}\right) \text{d}s \nonumber\\
&\hspace{0.75in} +\int_{\Gamma}\left(\Delta u^i(\cdot, \Hat{x})\frac{\partial u^i(\cdot, d)}{\partial\nu}-u^i(\cdot,d)\frac{\partial \Delta u^i(\cdot,\Hat{x})}{\partial\nu}\right) \text{d}s=0.
\end{align}
Using (\ref{Green's_2nd_id}) to the scattered fields $u^s(\cdot,d)$ and $u^s(\cdot,\Hat{x})$ in in the exterior domain $B_R\setminus\overline{D}$, where $B_R$ denotes the open ball of radius $R>0$, centered at the origin, chosen so that 
$D \subset B_R$, we have $I_{\partial B_R}-I_{\Gamma}=0$, where
\begin{align}\label{scattered_wave_rep}
I_{\Gamma}&\coloneqq\int_{\Gamma}\left(u^s(\cdot,\Hat{x})\frac{\partial \Delta u^s(\cdot,d)}{\partial\nu}-\Delta u^s(\cdot,d)\frac{\partial u^s(\cdot,\Hat{x})}{\partial\nu}\right)\text{d}s \nonumber\\
&\hspace{0.75in} +\int_{\Gamma}\left(\Delta u^s(\cdot, \Hat{x})\frac{\partial u^s(\cdot, d)}{\partial\nu}-u^s(\cdot,d)\frac{\partial \Delta u^s(\cdot,\Hat{x})}{\partial\nu}\right)\text{d}s,
\end{align}
and
\begin{align}\label{disc_integral}
I_{\partial B_R} &\coloneqq\int_{\partial B_R}\left(u^s(\cdot,\Hat{x})\frac{\partial \Delta u^s(\cdot,d)}{\partial\nu}-\Delta u^s(\cdot,d)\frac{\partial u^s(\cdot,\Hat{x})}{\partial\nu}\right) \text{d}s \nonumber\\
&\hspace{0.75in} +\int_{\partial B_R}\left(\Delta u^s(\cdot, \Hat{x})\frac{\partial u^s(\cdot, d)}{\partial\nu}-u^s(\cdot,d)\frac{\partial \Delta u^s(\cdot,\Hat{x})}{\partial\nu}\right)\text{d}s.
\end{align}
We now wish to verify that
\[
I_{\partial B_R}\to 0 \quad\text{as } \quad R \to \infty.
\]

Applying the decomposition of the scattered field given in \eqref{1.6}, and substituting these expressions into the boundary integral $I_{\partial B_R}$ in (\ref{disc_integral}), we deduce 
\begin{align*}
I_{\partial B_R}=(J_1-J_2)+(J_3-J_4),
\end{align*}
where
\begin{align*}
J_1&=\kappa^2\int_{\partial B_R}\left(u_{\text{H}}^s(\cdot,\Hat{x})\frac{\partial u_{\text{M}}^s(\cdot,d)}{\partial\nu}-u_{\text{M}}^s(\cdot,\Hat{x})\frac{\partial u_{\text{H}}^s(\cdot,d)}{\partial\nu}\right)\text{d}s,\\
J_2&=\kappa^2\int_{\partial B_R}\left(u_{\text{M}}^s(\cdot,d)\frac{\partial u_{\text{H}}^s(\cdot,\Hat{x})}{\partial\nu}-u_{\text{H}}^s(\cdot,d)\frac{\partial u_{\text{M}}^s(\cdot,\Hat{x})}{\partial\nu}\right) \text{d}s,\\
J_3&=\kappa^2\int_{\partial B_R}\left(u_{\text{M}}^s(\cdot,\Hat{x})\frac{\partial u_{\text{H}}^s(\cdot,d)}{\partial\nu}-u_{\text{H}}^s(\cdot,\Hat{x})\frac{\partial u_{\text{M}}^s(\cdot,d)}{\partial\nu}\right) \text{d}s,\\
J_4&=\kappa^2\int_{\partial B_R}\left(u_{\text{H}}^s(\cdot,d)\frac{\partial u_{\text{M}}^s(\cdot,\Hat{x})}{\partial\nu}-u_{\text{M}}^s(\cdot,d)\frac{\partial u_{\text{H}}^s(\cdot,\Hat{x})}{\partial\nu}\right)\text{d}s.
\end{align*}
For each term $J_j$ using the asymptotic behavior of the fields, we note that the $L^2(\partial B_R)$--norm of $u^s_{\text{H}}$ and $\partial_r u^s_{\text{H}}$ remain bounded as $ R \to \infty$. Moreover,
\begin{align*}
u_{\text{M}}^s = \mathcal{O}\left(\frac{\mathrm{e}^{-\kappa R}}{\sqrt{R}}\right), \quad  \partial_r u_{\text{M}}^s= \mathcal{O}\left(\frac{\mathrm{e}^{-\kappa R}}{\sqrt{R}}\right),\quad\text{as } R\to\infty,
\end{align*}
and by applying the Cauchy--Schwarz inequality, we deduce that
\begin{align*}
J_j\to 0\quad\text{as } R\to \infty,\quad j=1,2,3,4.
\end{align*}
Therefore, we conclude that  $I_{\partial B_R}\to 0$ as $R\to\infty$. Consequently, it follows that $I_{\Gamma}=0$.

Noting that the incident plane wave satisfies the Helmholtz equation, i.e., $\Delta u^i=-\kappa^2u^i$, and substituting this identity into the far-field pattern representation given by (\ref{Green's_far_field}), we obtain
\begin{align} \label{far_field_rep1}
&\frac{1}{2\kappa^2}\int_{\Gamma}\left(-u^s(\cdot,\Hat{x})\frac{\partial \Delta u^i(\cdot,d)}{\partial\nu}+\Delta u^i(\cdot,d)\frac{\partial u^s(\cdot,\Hat{x})}{\partial\nu}\right) \text{d}s \nonumber\\
&\hspace{0.5in} -\frac{1}{2\kappa^2}\int_{\Gamma}\left(\Delta u^s(\cdot,\Hat{x})\frac{\partial u^i(\cdot,d)}{\partial\nu}-u^i(\cdot,d)\frac{\partial \Delta u^s(\cdot,\Hat{x})}{\partial\nu}\right)\text{d}s=\frac{\sqrt{8\kappa \pi}}{\mathrm{e}^{\mathrm{i}\pi/4}}u^{\infty}(-d,\Hat{x}).
\end{align}
Similarly, we have
\begin{align}\label{far_field_rep2}
&\frac{1}{2\kappa^2}\int_{\Gamma}\left(-u^s(\cdot,d)\frac{\partial \Delta u^i(\cdot,\Hat{x})}{\partial\nu}+\Delta u^i(\cdot,\Hat{x})\frac{\partial u^s(\cdot,d)}{\partial\nu}\right) \text{d}s \nonumber\\
&\hspace{0.5in}-\frac{1}{2\kappa^2}\int_{\Gamma}\left(\Delta u^s(\cdot,d)\frac{\partial u^i(\cdot,\Hat{x})}{\partial\nu}-u^i(\cdot,\Hat{x})\frac{\partial \Delta u^s(\cdot,d)}{\partial\nu}\right)\text{d}s=\frac{\sqrt{8\kappa \pi}}{\mathrm{e}^{\mathrm{i}\pi/4}}u^{\infty}(-\Hat{x},d).
\end{align}
Now subtracting (\ref{far_field_rep1}) from the sum of (\ref{plane_wave_rep}), (\ref{scattered_wave_rep}), and (\ref{far_field_rep2}), we obtain
\begin{align*}
\frac{8\kappa \pi}{\mathrm{e}^{\mathrm{i}\pi/4}}(u^{\infty}(-d,\Hat{x})-u^{\infty}(-\Hat{x},d))&=\frac{1}{2\kappa^2}\int_{\Gamma}\left(-u(\cdot,d)\frac{\partial \Delta u(\cdot,\Hat{x})}{\partial\nu}+\Delta u(\cdot,\Hat{x})\frac{\partial u(\cdot,d)}{\partial\nu}\right)\text{d}s\\
&\quad -\frac{1}{2\kappa^2}\int_{\Gamma}\left(\Delta u(\cdot,d)\frac{\partial u(\cdot,\Hat{x})}{\partial\nu}-u(\cdot,\Hat{x})\frac{\partial \Delta u(\cdot,d)}{\partial\nu}\right)\text{d}s
=0.
\end{align*}
Notice that we have used the boundary conditions for the total field, i.e., $u=0, \partial_{\nu}u=0$ on $\Gamma$, which completes the proof of the reciprocity relation.
\end{proof}

\section{The linear sampling method}\label{section 4}

In this section, we investigate the application of the LSM to reconstruct the scatterer $D$ from far-field measurements. The LSM is designed to construct an indicator function that determines whether a sampling point $ z \in \mathbb{R}^2$ lies inside or outside the cavity $D$. This is achieved by approximately solving the far-field equation: for each $z$, find a function $g_z\in L^2(\mathbb S^1)$ satisfying
\begin{align}\label{far_field_eqn}
\mathcal{F}g_z= \Phi^{\infty}(\cdot \, ,z), \quad \text{ where }\, \Phi^{\infty}(\Hat{x},z)=-\frac{1}{2\kappa^2}\frac{\text{e}^{\mathrm{i}\pi/4}}{\sqrt{8\pi\kappa}}\text{e}^{-\mathrm{i}\kappa\Hat{x}\cdot z},\quad z\in \mathbb R^2. 
\end{align}
Here,  $\Phi^{\infty}(\Hat{x},z)$ denotes the far-field pattern of the outgoing fundamental solution to the biharmonic operator $\Delta^2-\kappa^4$, corresponding to a point source located at $z\in\mathbb R^2$. This fundamental solution is explicitly given by
\[
\Phi(x,z)=\frac{1}{2\kappa^2} \Big[\Phi_{\rm H}(x,z; \kappa)-\Phi_{\rm M}(x,z; \kappa) \Big],\quad x\neq z.
\]
The mapping $z \mapsto \| g_z \|_{L^2(\mathbb{S}^1)}$ serves as the LSM indicator function: it tends to be large when $z\notin D$ and relatively small when $z\in D$, thereby allowing for the identification of the cavity's support.

The LSM relies on a key analytical property of the far-field operator $\mathcal F$, i.e., it's injective with a dense range, under a suitable assumption on the wavenumber $\kappa$. This property of $\mathcal{F}$ guarantees the approximate solvability of the far-field equation and is essential for the validity of the LSM.    

Following earlier studies, we derive a factorization of the far-field operator $\mathcal{F}$ to facilitate its analytical investigation. We introduce the auxiliary operator $\mathcal{H} : L^2(\mathbb S^1)\to H^{3/2}(\Gamma)\times H^{1/2}(\Gamma)$ defined by 
\begin{align}\label{herglotzOp}
\mathcal{H}g \coloneqq \big( v_g |_{\Gamma}, \partial_{\nu}v_g |_{\Gamma} \big)^\top,
\end{align}
where 
\begin{align}\label{herglotzDef}
    v_g(x)\coloneqq \int_{\mathbb S^1} \mathrm{e}^{\mathrm{i}\kappa x\cdot d}g(d) \text{d}s(d),\quad x\in \mathbb R^2, 
\end{align}
is the Herglotz wave function. The operator $\mathcal H$ is called the Herglotz wave operator. By the linearity of the direct scattering problem, the far-field operator $\mathcal{F}g$ can be interpreted as the far-field pattern corresponding to the incident field $v_g$, i.e., with boundary data $-\mathcal{H}g$ in the direct scattering problem (\ref{1.7})--(\ref{1.9}). 

To complete the factorization of the far-field operator, we introduce an additional auxiliary operator. As established in \cite{bourgeois2020well}, for any pair $(h_1 , h_2) \in H^{3/2}(\Gamma)\times H^{1/2}(\Gamma)$, the boundary value problem
\begin{align}\label{4.6}
    \begin{dcases}
        \Delta^2 w-\kappa^4w=0\quad \text{in }\mathbb R^2\setminus\overline{D},\\
        w= h_1,\quad \partial_{\nu}w=h_2\quad \text{on }\Gamma,\\
        \lim_{r\to \infty}\sqrt{r}(\partial_rw-\mathrm{i}\kappa w)=0
    \end{dcases}
\end{align}
is well-posed for $w\in H_{\text{loc}}^2(\mathbb R^2\setminus\overline{D})$. Indeed, we have the stability estimate 
\begin{align*}
\|w\|_{H^2(B_R\setminus\overline{D})}\leq C \big(\|h_1\|_{H^{3/2}(\Gamma)}+\|h_2\|_{H^{1/2}(\Gamma)}\big),
\end{align*}
where $B_R$ is a ball of sufficiently large radius $R$ centered at the origin such that $D\subset B_R$, and $C>0$ is a constant depending only on $R$. Based on this, we define the data-to-pattern operator $ \mathcal{G}: H^{3/2}(\Gamma)\times H^{1/2}(\Gamma)\to L^2(\mathbb S^1)$, given by 
\begin{align}\label{data2pattern}
 \mathcal{G}( h_1, h_2)^\top=w^{\infty}, 
\end{align}
which maps the boundary data $(h_1, h_2)^\top\in H^{3/2}(\Gamma)\times H^{1/2}(\Gamma)$ to the corresponding far-field pattern $w^{\infty}$ of the radiating solution $w$ to the boundary value problem \eqref{4.6}. By the superposition principle, the far-field operator $\mathcal{F}$ admits the factorization $\mathcal{F} = -\mathcal{G} \mathcal{H}$. The theoretical foundation of the LSM is based on characterizing the cavity $D$ in terms of the range of the auxiliary operator $\mathcal{G}$. 

In a subsequent lemma, we assume that the wavenumber $\kappa$ is not an eigenvalue of the clamped transmission eigenvalue problem: find $(p,q)\in H^1(\mathbb R^2\setminus\overline{D})\times H^1(D)$ such that 
\begin{align} \label{eigenvalue_problem} 
    \begin{dcases}
         \Delta p - \kappa^2 p =0 \quad  \text{in }\mathbb R^2\setminus\overline{D}, \\
          \Delta q + \kappa^2 q =0 \quad \text{in }{D},\\
        p+q=0,\quad  \partial_{\nu} (p+q)=0\quad\text{on }\Gamma,
    \end{dcases}
\end{align}
with the additional condition that $p$ decays exponentially as $r\to\infty$. We refer to \cite{HLK-tep} for a detailed discussion of the above clamped transmission eigenvalue problem \eqref{eigenvalue_problem}. 
    
In addition, we establish a key lemma essential for the applicability of the LSM. Recall that the LSM constructs an indicator function for the cavity $D$ via the far-field equation. Below, we present a range characterization of the scatterer $D$ in terms of the data-to-pattern operator.

\begin{lemma}\label{rangethm}
For the operator $\mathcal{G}\,:\, H^{3/2}(\Gamma)\times H^{1/2}(\Gamma)\to L^2(\mathbb S^1)$ defined by \eqref{data2pattern}, the clamped cavity $D$ admits the following range characterization:
\[
\Phi^{\infty}(\cdot \, , z) \in \mathrm{Range}(\mathcal{G})\iff z\in D.
\]
\end{lemma}

\begin{proof}
To prove the claim, we first assume that $z \in D$ and let  
$w_z(x) = \Phi(x,z) \in \mathbb{R}^2 \setminus \overline{D}$, where $\Phi(x,z)$ is the fundamental solution to the operator $\Delta^2-\kappa^4$ with a source at $z$. Therefore, we can set the boundary data 
$$(h^z_1, h^z_2)^\top = (w_z|_{\Gamma}, \partial_{\nu}w_z|_{\Gamma})^\top.$$
Notice that $\Phi(x,z)$ is smooth in the region $\mathbb{R}^2 \setminus \overline{D}$. By the trace theorem, we have $(h^z_1, h^z_2)^\top \in H^{3/2}(\Gamma) \times H^{1/2}(\Gamma)$. The far-field pattern of \(w\) is given by
\[
w_z^\infty(\hat{x}) = -\frac{1}{2\kappa^2} \frac{\text{e}^{\text{i}\pi/4}}{\sqrt{8\pi \kappa}} \text{e}^{-\text{i}\kappa \hat{x} \cdot z}, \quad \hat{x} \in \mathbb{S}^1, 
\]
which coincides with \(\Phi^\infty (\cdot \, , z) \). Therefore, we obtain 
\[
\mathcal{G}( h_1, h_2)^\top = w_z^\infty = \Phi^\infty (\cdot \, , z),
\]
which implies that $\Phi^\infty (\cdot \, , z) \in\mathrm{Range}(\mathcal{G})$.

Consider $z \in \mathbb{R}^2 \setminus \overline{D}$. Assume, for contradiction, that there is a pair $(h^z_1, h^z_2)^\top \in H^{3/2}(\Gamma) \times H^{1/2}(\Gamma)$ such that 
\[ 
\mathcal{G}(h^z_1, h^z_2)^\top = \Phi^\infty(\cdot, z).
 \]
Then, there exists a solution $w_z$ to \eqref{4.6} whose far-field pattern satisfies $w_z^\infty = \Phi^\infty (\cdot \, , z)$. By definition, the far-field pattern of $\Phi(\cdot \, , z)$ coincides, up to a multiplicative constant, with that of $\Phi_{{\text{H}}}(\cdot, z)$, the fundamental solution of the Helmholtz equation. By applying Rellich's lemma, we conclude that
\[
w_{z,\text{H}} = \Phi_{{\text{H}}}(\cdot \,, z) \quad \text{in} \,  \mathbb{R}^2 \setminus \overline{D},
\]
where $w_{z,\text{H}}$ is the Helmholtz component of the solution $w_z$. However, this leads to a contradiction since 
$|w_{z,\text{H}}(x)| < \infty$ and $|\Phi_{{\text{H}}}(x ,  z)| \to \infty$ as $x\to z$. Therefore, the claim follows.
\end{proof}

Consequently, the range characterization of the cavity $D$ in terms of the operator $\mathcal{G}$, as established in Lemma \ref{rangethm}, does not directly translate into an analogous characterization in terms of the range of the far-field operator $\mathcal{F}$. However, by the factorization of the far-field operator, we have  ${\rm Range}(\mathcal{F})\subset {\rm Range}(\mathcal{G})$, a property that is important in our subsequent analysis. To proceed, we now turn our attention to the operator $\mathcal{H}$ defined in \eqref{herglotzOp}.

\begin{lemma}\label{Herglotz_wave}
The Herglotz wave operator $\mathcal{H}: L^2(\mathbb S^1)\to H^{3/2}(\Gamma)\times H^{1/2}(\Gamma)$ defined by \eqref{herglotzOp} is compact and injective. 
\end{lemma}

\begin{proof}
To prove compactness, notice that the corresponding Herglotz wave functions are smooth solutions to the Helmholtz equation in $\mathbb{R}^2$. Therefore, we have that $v_g \in H^3_{loc}(\mathbb{R}^2)$, which implies  
\[
{\rm Range}(\mathcal{H}) \subset H^{5/2}(\Gamma) \times H^{3/2}(\Gamma). 
\]
The compactness of $\mathcal H$  then follows directly from standard Sobolev embedding theorems. 

To establish injectivity, suppose that $\mathcal{H}g = 0$. 
Then the associated Herglotz wave function $v_g$ satisfies $v_g = 0$ and $\partial_\nu v_g=0$ on $\Gamma$. 
By the unique continuation principle for solutions to the Helmholtz equation, it follows that $v_g = 0$ in $D$. This implies that $g = 0$, which proves that $\mathcal{H}$ is injective.
\end{proof}

Next, we present a result concerning a key analytical property of the far-field operator. Recall the factorization
\begin{align}\label{FF-factorization}
\mathcal{F} = - \mathcal{G} \mathcal{H},
\end{align}
where the operators $\mathcal{H}$ and $\mathcal{G}$ are defined in \eqref{herglotzOp} and \eqref{data2pattern}, respectively.  We also recall the reciprocity relation established in Theorem \ref{Reciprocity Thm}. Combining these observations with Theorem \ref{Herglotz_wave}, we arrive at the following result.

\begin{theorem}\label{FFanalysis}
The far-field operator $\mathcal{F}: L^2(\mathbb S^1)\to L^2(\mathbb S^1)$ associated with \eqref{1.3} is compact and injective with a dense range provided that the wavenumber $\kappa$ is not a clamped transmission eigenvalue for \eqref{eigenvalue_problem}. 
\end{theorem}

\begin{proof}
It follows from the factorization \eqref{FF-factorization} and the compactness of $\mathcal{H}$ established in Theorem \ref{Herglotz_wave} that the far-field operator $\mathcal{F}$ is also compact.

To establish injectivity, suppose that $\mathcal{F}g=0$ for some $g \in L^2(\mathbb S^1)$. Since $\mathcal{F}g$ is the far-field pattern associated with \eqref{1.3} for the incident field $u^i = v_g$, where $v_g$ is the Herglotz wave function defined in \eqref{herglotzDef}, it follows that the Helmholtz component of the solution satisfies $u^s_{\text{H}} = 0$ in $\mathbb{R}^2 \setminus \overline{D}$. Consequently, the modified Helmholtz component of the total field, given by
\[
(p,q) = (u^s_{\text{M}}, v_g) \in H^1(\mathbb{R}^2 \setminus \overline{D}) \times H^1(D)
\]
solves the eigenvalue problem \eqref{eigenvalue_problem}. By the assumption that $\kappa$ is not a clamped transmission eigenvalue, it follows that $v_g=0$ in $D$, which in turn implies $g=0$. This proves the injectivity of $\mathcal{F}$.

Lastly, regarding the density of the range, we note that, by the reciprocity relationship established in Theorem \ref{Reciprocity Thm}, simple calculations yield
\[
(\mathcal{F}^* f)(d) = \int_{\mathbb{S}^1} \overline{u^{\infty}(\hat{x}, d)} f(-d) \mathrm{d}s(d). 
\]
It follows that $\mathcal{F}^*$ is injective if and only if $\mathcal{F}$ is injective. Since the far-field operator is linear and bounded, we have
\[
\overline{\mathrm{Range}(\mathcal{F})} = \mathrm{Null}(\mathcal{F}^*)^\perp.
\]
Therefore, the injectivity of $\mathcal{F}^*$ implies that $\overline{\mathrm{Range}(\mathcal{F})}=L^2(\mathbb S^1)$, establishing the density of the range and proving the claim.
\end{proof} 

The following result characterizes the behavior of the indicator function in the LSM. In essence, this result states that the approximate solution to the far-field equation \eqref{far_field_eqn} remains bounded when the sampling point lies inside the scatterer $D$. This property provides a practical computational criterion for reconstructing the scatterer from the far-field operator.

\begin{theorem}\label{Theorem_LSM}
Assume that the wavenumber $\kappa$ is not an eigenvalue of the clamped transmission eigenvalue problem given in \eqref{eigenvalue_problem}. Then, for any $z \in \mathbb{R}^2 \setminus \overline{D}$ and any sequence $\{g_z^\alpha\} \subset L^2(\mathbb{S}^1)$ satisfying
\begin{align}\label{LSM-1}
\lim_{\alpha \to 0} \|\mathcal{F} g_z^\alpha - \Phi^\infty(\cdot, z)\|_{L^2(\mathbb{S}^1)} = 0,
\end{align}
it follows that
\[
\lim_{\alpha \to 0} \|g_z^\alpha\|_{L^2(\mathbb{S}^1)} = \infty.
\]
\end{theorem}

\begin{proof} 
We begin by assuming $z\in \mathbb{R}^2\setminus\overline{D}$. Since $\kappa$ is not an eigenvalue of the clamped transmission eigenvalue problem, the operator $\mathcal{F}$ is both compact and one-to-one in $L^2(\mathbb S^1)$, with a range that is dense in the space. Therefore, the far-field equation \eqref{far_field_eqn} admits an approximate solution, which may be obtained, for instance, via Tikhonov regularization, such that
\[
\|\mathcal{F}g_z^{\alpha_j}-\Phi^{\infty}(\cdot \, ,z)\|_{L^2(\mathbb S^1)} \leq \frac{1}{j} \quad \text{ as }\, \alpha_j \to 0,
\]
as $j \to \infty$. Define $g_j := g_z^{\alpha_j}$. We claim that the sequence $\{g_j\}$ cannot be bounded in $L^2(\mathbb{S}^1)$.

Suppose, for the sake of contradiction, that $\|g_j\|_{L^2(\mathbb{S}^1)}$ is bounded for all $j \in \mathbb{N}$. Then,  there exists a subsequence, still denoted by $g_j$, that converges weakly to some $g \in L^2(\mathbb{S}^1)$. Since  $\mathcal{H}$ is compact, we have $\mathcal{H}g_j \to  \mathcal{H}g $ in $H^{3/2}(\Gamma)\times H^{1/2}(\Gamma)$ as $j\to\infty$. As $\mathcal{G}$ is bounded, it follows that $\mathcal{G}\mathcal{H}g_j \to  \mathcal{G}\mathcal{H}g $ in $L^2(\mathbb S^1)$ as $j\to\infty$. Using the factorization \eqref{FF-factorization} and the convergence in \eqref{LSM-1}, we obtain 
\[
\mathcal{F}g= -\mathcal{G}\mathcal{H}g = \Phi^{\infty}(\cdot \, ,z),\quad z\in \mathbb{R}^2\setminus\overline{D}.
\]
This contradicts Theorem \ref{rangethm}, and therefore $\|g_z^{\alpha}\|_{L^2(\mathbb S^1)}\to\infty$ as $\alpha \to 0$.
\end{proof}

The LSM reformulates the problem of determining the shape of the cavity $D$ as the calculation of the indicator function $g^{\alpha}_z$, as described in Theorem \ref{Theorem_LSM}. The overall computational steps are summarized in Algorithm 1.

\begin{algorithm}[H]
\caption{Linear Sampling Method (LSM)}
\begin{algorithmic}[1]
\STATE Choose a cutoff parameter $\zeta>0$, and select a mesh $\mathcal{M}$ of sampling points in a region $\Omega$ that contains the cavity $D$;
\STATE For each sampling point $z\in \mathcal{M}$, compute an approximate solution $g^{\alpha}_z$  to the far-field equation \eqref{far_field_eqn} using Tikhonov regularization in conjunction with the Morozov discrepancy principle;
\STATE Classify the sampling point $z$ as inside the cavity $D$ if $1/\norm{g^{\alpha}_z}_{L^2(\mathbb S^1)}>\zeta$, and as outside $D$ if $1/\norm{g^{\alpha}_z}_{L^2(\mathbb S^1)}\leq \zeta$. 
\end{algorithmic}
\end{algorithm}

\section{The extended sampling method}\label{section_ESM}

In this section, we adapt the extended sampling method (ESM) to the inverse scattering problem of locating a clamped cavity based on far-field measurements from one or a few incident directions. Specifically, the goal is to determine the location of the cavity $D$ from $u^{\infty}(\Hat{x},d)$ not identically zero, where $(\Hat{x},d)\in \mathbb{S}^1\times \mathbb{S}^1_{\text{inc}}$ for one or multiple wavenumbers. Here, $ \mathbb{S}^1_{\text{inc}} \subsetneq \mathbb{S}^1$ denotes the set of incident directions. For example, if $\mathbb{S}^1_{\text{inc}}=\{d\}$, the scattering data consist of the far-field pattern generated by a single incident wave. In contrast, the full aperture case corresponds to the far-field data $u^{\infty}(\Hat{x},d)$  for all $\Hat{x},d\in\mathbb S^1$, i.e., $ \mathbb{S}^1_{\text{inc}}=\mathbb{S}^1$. In this case, the clamped cavity $D$ can be uniquely determined from the full aperture far-field data. The ESM is designed to recover the location of the clamped cavity using limited aperture data.

We begin by presenting the justification of the ESM for the inverse clamped scattering problem with a single incident wave. Consider the case where the far-field pattern $u^{\infty}(\Hat{x},d)$ is known corresponding to one fixed incident direction $d$. Let $B=B_R(0)\subset \mathbb R^2$ denote a sound-soft disk centered at the origin with sufficiently large radius $R$. For any sampling point $z \in \mathbb{R}^2$, we define
\[
B_z=B(z,R)\coloneqq \{x+z\,|\,x\in B,\, z\in\mathbb R^2\},
\]
and let 
$U^s_{B_z}(x ,\hat{y})$, with $x \in \mathbb{R}^2 \setminus \overline{B_z}$, denote the solution to acoustic sound soft scattering problem
\begin{align}\label{aux-SSprob}
\begin{dcases}
\Delta U_{B_z}^s+\kappa^2 U_{B_z}^s=0\quad \text{in }\mathbb R^2\setminus\overline B_z,\\
U_{B_z}^s=-\text{e}^{\text{i}\kappa x \cdot \hat{y}}\quad \text{on }\partial B_z,\\
\lim_{r\to\infty}\sqrt{r}\left(\partial_rU_{B_z}^s-{\rm i}\kappa U_{B_z}^s\right)=0.
\end{dcases}
\end{align}
By applying the method of separation of variables (see, e.g., \cite{Liu_2018}) to the case of a disk centered at the origin, the far-field pattern of $U^s_{B}(x ,\hat{y})$ is given by
\begin{align}\label{aux-ffpattern1}
 U_B^{\infty}(\hat{x},\hat{y}) = -\text{e}^{-\text{i} \pi/4} \sqrt{\frac{2}{\pi \kappa}} \left[ \frac{J_0(\kappa R)}{H^{(1)}_0(\kappa R)} +2\sum_{n=1}^{\infty} \frac{J_n(\kappa R)}{H^{(1)}_n(\kappa R)} \cos \big(n (\theta_x-\theta_y) \big) \right], \quad \Hat{x}\in\mathbb S^1,
\end{align}
where $J_n$ and $H^{(1)}_n$ refer, respectively, to the Bessel and Hankel functions of the first kind of order $n$, and $\theta_x, \theta_y$ are the observation angle for $\Hat{x}$ and the incident angle for $\Hat{y}$, respectively. By \cite{Liu_2018}, the far-field pattern of the solution $U^s_{B_z}$ for a disk centered at $z \in \mathbb{R}^2$ can be expressed as
\begin{align}\label{aux-ffpattern2}
U_{B_z}^{\infty}(\Hat{x},\Hat{y})=\text{e}^{{\rm i}\kappa z\cdot (\Hat{y}-\Hat{x})}U_B^{\infty}(\Hat{x},\Hat{y}),\quad \Hat{x}\in\mathbb S^1.
\end{align}

To construct a sampling method using the precomputed far-field pattern $U_{B_z}^{\infty}(\Hat{x},\Hat{y})$ together with the measured far-field data $u^\infty (\hat{x} , d)$, let $\mathcal{M}$  denote a sampling domain containing the cavity $D$. Note that $U_{B_z}^{\infty}(\Hat{x},\Hat{y})$, obtained from \eqref{aux-ffpattern1}--\eqref{aux-ffpattern2}, is independent of the actual scatterer $D$. For each sampling point $z \in \mathcal{M}$, let $ \mathcal{F}_{B_z} : L^2(\mathbb S^1)\to L^2(\mathbb S^1)$ be the modified far-field operator, defined by 
\begin{align*}
\mathcal{F}_{B_z}g(\Hat{x})= \int_{\mathbb S^1}U_{B_z}^{\infty}(\Hat{x},\Hat{y})g(\Hat{y}) \text{d}s(\Hat{y}),\quad \Hat{x}\in \mathbb S^1, 
\end{align*}
where $U_{B_z}^{\infty}(\Hat{x},\Hat{y})$ denotes the far-field pattern corresponding to the scattering of a sound-soft disk
$B_z$ by an incident plane wave from direction $\hat{y}$. Using the modified far-field operator $\mathcal{F}_{B_z}$, we formulate the modified far-field equation associated with the ESM. Specifically, for a fixed $d\in\mathbb S^1$, we seek a function $g\in L^2(\mathbb S^1)$ satisfying
\begin{align}\label{ESM_FF_eqn}
(\mathcal{F}_{B_z}g)(\Hat{x})=u^{\infty}(\Hat{x},d),\quad \Hat{x} \in\mathbb S^1. 
\end{align}

It is advantageous to place the given data on the right-hand side of the far-field equation. To analyze this equation, we introduce auxiliary operators that enable the factorization of the modified far-field operator. Following \cite{Liu_2018}, we define $\mathcal G_{B_z} : H^{1/2}(\partial B_z)\to L^2(\mathbb S^1)$ by 
\begin{align*}
\mathcal{G}_{B_z}f\coloneqq W^{\infty},
\end{align*}
where $W^{\infty}$ denotes the far-field pattern of $W^s$ that satisfies 
\begin{align*}
\begin{dcases}
\Delta W^s+\kappa^2 W^s=0\quad \text{in }\mathbb R^2\setminus\overline B_z,\\
W^s=f\quad \text{on }\partial B_z,
\end{dcases}
\end{align*}
together with the Sommerfeld radiation condition. This definition is equivalent to replacing the plane wave $-\text{e}^{\text{i}\kappa x \cdot \hat{y}}$ in \eqref{aux-SSprob} with a general boundary function $f \in H^{1/2}(\partial B_z)$.

We also define the Herglotz operator $\mathcal{H}_{B_z} : L^2(\mathbb S^1)\to H^{1/2}(\partial B_z)$ associated with \eqref{aux-SSprob} as
\begin{align*}
\mathcal{H}_{B_z}g = v_{g_z}|_{\partial B_z}, 
\end{align*}
where $v_g$ is the Herglotz wave function with density $g$, defined in \eqref{herglotzDef}. Then, as in the previous section, the modified far-field operator admits the factorization
\begin{align*}
\mathcal{F}_{B_z}=-\mathcal G_{B_z}\mathcal{H}_{B_z}.
\end{align*}
We note that this modified far-field operator is associated with an acoustic scattering problem, which differs fundamentally from the biharmonic scattering problem under consideration here.

The main reason for considering using the ESM method which was developed for acoustic scattering is due to the fact that $u^{\infty}(\Hat{x},d)=u_{\text{H}}^{\infty}(\Hat{x},d)$. This implies that even though the given far-field data corresponds to a biharmonic scattering problem, we only retain the propagating part of the solution in the far-field. 

The following theorem provides a theoretical justification for how the solution of equation \eqref{ESM_FF_eqn} can be used to characterize the location of the clamped cavity $D$.

\begin{theorem}\label{ESM_THM}
Let $B_z$ be a disk of radius $R$ centered at a sampling point $z$, and let $D$ be a clamped cavity in a thin elastic plate. Suppose that $\kappa^2$ is not a Dirichlet eigenvalue of $-\Delta$ in $B_z$. Then, for any $d \in \mathbb{S}^1$, the modified far-field equation admits the following properties:
\begin{enumerate}
\item[(i)] If $D\subset B_z$, then for every $\epsilon>0$, there exists a function $g_z^{\alpha}\in L^2(\mathbb S^1)$ such that
\begin{align}\label{ESM_FF_ineq}
\displaystyle\lim_{\alpha\to 0} \norm{\mathcal F_{B_z}g_z^{\alpha}-u^{\infty}( \cdot \,  , d)}_{L^2(\mathbb S^1)} = 0
\end{align}
and the associated Herglotz wave function $v_{g_z^{\alpha}}$ converges to $v \in H^1(B_z)$ that is a solution to the Helmholtz equation in $B_z$, where $v=-u_{\text{H}}^s ( \cdot \,  , d) $ on $\partial B_z$, as $\alpha\to 0$.
\item[(ii)] If $D\cap B_z=\emptyset$, then for any $g_z^{\alpha}$ satisfying \eqref{ESM_FF_ineq}, it holds that
\begin{align}\label{ESM_FF_2}
\lim_{\alpha\to 0}\norm{g_z^{\alpha}}_{L^2(\mathbb S^1)}&=\infty.
\end{align}
\end{enumerate}
\end{theorem}

\begin{proof}
First, we prove that \eqref{ESM_FF_ineq} holds. Assume that $D\subset B_z$ and define $f = u_{\text{H}}^s ( \cdot \,  , d) |_{\partial B_z}$, the trace of the propagating part of the scattered field, which belongs to $H^{1/2}(\partial B_z)$ since $D\subset B_z$. By the definition of $\mathcal G_{B_z}$, the function $f$ satisfies 
$$\mathcal G_{B_z}f=u_{\text{H}}^{\infty} ( \cdot \,  , d) =u^{\infty}( \cdot \,  , d).$$
This follows from the fact that $u_{\text{H}}^s( \cdot \,  , d)$ solves \eqref{aux-SSprob} with Dirichlet data $f = u_{\text{H}}^s( \cdot \,  , d) |_{\partial B_z}$. Hence, for any $d \in \mathbb{S}^1$, we have $u^{\infty} ( \cdot \,  , d) \in \mathrm{Range} (\mathcal{G}_{B_z})$. 

Since $\kappa^2$ is not a Dirichlet eigenvalue of $-\Delta$ in $B_z$, it follows from \cite[Lemma 3.1]{Liu_2018} that the operator $\mathcal{H}_{B_z}$ has dense range. Therefore, there exists $g_z^{\alpha}\in L^2(\mathbb S^1)$ such that
\begin{align*}
\lim_{\alpha\to 0} \norm{\mathcal H_{B_z}g_z^{\alpha}+u_{\text{H}}^s ( \cdot \,  , d) }_{H^{1/2}(\partial B_z)} = 0.
\end{align*}
Consequently,
\begin{align*}
\norm{\mathcal F_{B_z}g_z^{\alpha}-u^{\infty} ( \cdot \,  , d) }_{L^2(\mathbb S^1)}&= \norm{\mathcal G_{B_z}(-\mathcal H_{B_z})g_z^{\alpha}-\mathcal G_{B_z}u_{\text{H}}^s( \cdot \,  , d)}_{L^2(\mathbb S^1)}\\
&\leq \norm{\mathcal G_{B_z}}\cdot \norm{\mathcal H_{B_z}g_z^{\alpha}+u_{\text{H}}^s( \cdot \,  , d)}_{H^{1/2}(\partial B_z)}.
\end{align*}
By the well-posedness of the Dirichlet problem for the Helmholtz equation in $B_z$, we have that $v_{g_z^\alpha}$ converges in $H^1(B_z)$ to the unique solution $v \in H^1(B_z)$ of the Helmholtz equation with Dirichlet boundary condition $v=-u_{\text{H}}^s ( \cdot \,  , d)$ on $\partial B_z$.

Next, we prove \eqref{ESM_FF_2} by contradiction. Suppose that $D\cap B_z=\emptyset$ and that the modified far-field equation \eqref{ESM_FF_eqn} has an approximate solution $g_z^{\alpha}\in L^2(\mathbb S^1)$ satisfying $\|g_z^{\alpha}\|_{L^2(\mathbb S^1)}<\infty$. Then, there is a sequence of positive numbers such at $\alpha_n \to 0$ as $n\to \infty$ where the sequence $\{g_z^{\alpha_n}\}$ converges weakly to some $g_z \in L^2(\mathbb{S}^1)$. Since $\kappa^2$ is not a Dirichlet eigenvalue, the operator $\mathcal{F}_{B_z}$ has dense range; thus, there exists a density $g_z^{\alpha}$ for which \eqref{ESM_FF_ineq} holds. Consequently, $v_{g_z}^{\alpha_n}$ converges weakly to $v_{g_z}$ in $H^1_{\rm loc}(\mathbb R^2)$ as $\alpha_n \to 0$ and $n \to \infty$.

By well-posedness of the exterior sound-soft scattering problem, there exists a unique radiating solution $V^s\in H_{\text{loc}}^1(\mathbb R^2\setminus\overline{B_z})$ satisfying
\begin{align*}
\begin{dcases}
\Delta V^s+\kappa^2V^s=0\quad \text{in }\mathbb R^2\setminus\overline{B_z},\\
V^s|_{\partial B_z}=-v_{g_z},\\
\lim_{r\to\infty}\sqrt{r}(\partial_rV^s-{\rm i}\kappa V^s)=0.
\end{dcases}
\end{align*}
Let $V^{\infty}$ denote the far-field pattern of $V^s$. Because the modified far-field operator is compact,
\[
\mathcal{F}_{B_z}g_z^{\alpha_n} \to \mathcal{F}_{B_z}g_z = - \mathcal{G}_{B_z}(\mathcal H_{B_z} )g_z = V^{\infty}  \quad \text{ as }\, n \to \infty.
\]
From \eqref{ESM_FF_eqn}, we deduce $- \mathcal{G}_{B_z}v_{g_z}|_{\partial B_z} = V^{\infty}$, which implies 
$V^{\infty}=u^{\infty}( \cdot \,  , d)=u^{\infty}_{\text{H}}( \cdot \,  , d).$

By Rellich's lemma, it follows that
$V^s=u_{\text{H}}^s ( \cdot \,  , d)$ in $\mathbb R^2\setminus(\overline{D}\cup \overline{B_z})$. Define the function $W^s$ by
\begin{align*}
W^s \coloneqq \begin{dcases}
V^s\quad \quad\quad \text{ in }\mathbb R^2\setminus B_z,\\
u_{\text{H}}^s ( \cdot \,  , d) \quad \text{ in }B_z.
\end{dcases}
\end{align*}
Then $W^s\in H^1_{\rm loc}(\mathbb R^2)$ is a radiating solution to the Helmholtz equation in all of $\mathbb R^2$.  By the uniqueness of radiating solutions, we deduce that $W^s=0$. In particular, this implies $u_{\text{H}}^s=0$ in 
$B_z$, and by unique continuation, $u_{\text{H}}^s ( \cdot \,  , d) =0$ in $\mathbb R^2$. However, this contradicts the assumption that $u_{\text{H}}^{\infty} ( \cdot \,  , d)$ is not identically zero. 
\end{proof}

It is worth noting that, since we have control over the radius $R$ of the sampling disk, it can be selected so that $\kappa^2$ does not coincide with any Dirichlet eigenvalue of $-\Delta$ on $B_z$. Consequently, the ESM is applicable for all wavenumbers $\kappa$. In contrast, the LSM requires excluding wavenumbers for which the clamped transmission problem \eqref{eigenvalue_problem} admits nontrivial solutions.

An important step in implementing the ESM is selecting the radius $R$ of the sampling disk $B_z$. To this end, we adopt the multilevel ESM strategy proposed in \cite{Liu_2018} to determine an appropriate value of $R$, as summarized in Algorithm 2.

\begin{algorithm}[H]
\caption{Multilevel Extended Sampling Method (ESM)}
\begin{algorithmic}[1]
\STATE Choose an initial sampling radius $R$ sufficiently large, and construct a coarse sampling grid $\mathcal{M}$ such that the spacing between adjacent sampling points is approximately $R$;
\STATE For each $z \in \mathcal{M}$, compute the far-field data $U_{B_z}^{\infty}(\Hat{x},\Hat{y})$ for all $\Hat{x}, \Hat{y} \in \mathbb{S}^1$;
\STATE Using the ESM procedure, identify the global minimizer $z_0 \in \mathcal{M}$ of $\norm{g_z^{\alpha}}_{L^2(\mathbb{S}^1)}$ and use $B_{z_0}$ as an initial approximation of the support of $D$;
\FOR{$j=1,2,\dots$} 
\STATE Set $R_j = R / 2^j$ and construct a finer sampling grid $\mathcal{M}_j$ with spacing approximately $R_j$;
\STATE Determine the minimizer $z_j \in \mathcal{M}_j$ of $\norm{g_z^{\alpha}}_{L^2(\mathbb{S}^1)}$. If $z_j \notin B_{z_{j-1}}$, terminate the iteration and proceed to Step 8;
\ENDFOR
\STATE Return $z_{j-1}$ as the estimated location and $B_{z_{j-1}}$ as the approximate support of the clamped cavity $D$.
\end{algorithmic}
\end{algorithm}

In practical scenarios, far-field measurements can typically be collected at finitely many incident directions, 
\begin{align*}
u^{\infty}(\cdot \, ,d_j), \quad d_j\in \{d_1,d_2,\dots, d_J\}= \mathbb{S}^1_{\text{inc}} \subseteq \mathbb{S}^1.
\end{align*}
For each incident direction $d_j$, we consider the discrete system of equations
\begin{align}\label{approx_ESM_FF_eqn}
\left(\mathcal{F}_{B_z} g \right)(\cdot\,; d_j) = u^{\infty}(\cdot, d_j), \quad j = 1, \dots, J.
\end{align}
Denote by $g_z^{\alpha}(\hat{x}; d_j)$ the regularized solution to \eqref{approx_ESM_FF_eqn} corresponding to the incident direction $d_j$. Then, the indicator function $\mathcal{I}(z)$ for multiple incident directions $z \in \mathcal{M}$ is defined as
\begin{align}\label{ESM_ind_1}
z \mapsto \sum_{j=1}^J \| g_z^{\alpha}(\cdot; d_j) \|_{L^2(\mathbb{S}^1)}, \quad z \in \mathcal{M}.
\end{align}

Similarly, we can incorporate multiple frequency data. Let $g_z^{\alpha}(\cdot; d_j, \kappa_\ell)$ denote the regularized solution to the modified far-field equation at frequency $\kappa_\ell \in \{\kappa_\ell\}_{\ell=1}^L \subset \mathbb{R}_{>0}$. Then, the corresponding indicator function $\mathcal{I}(z)$ is defined by the mapping
\begin{align}\label{ESM_ind_2}
 z \longmapsto \sum_{\ell=1}^L\sum_{j=1}^J \norm{g_z^{\alpha} (\cdot \, ; d_j , \kappa_\ell)}_{L^2(\mathbb{S}^1)}, \quad z\in \mathcal M. 
\end{align}
This gives a method that detects the location of the scatterer from reduced far-field data.

\section{Numerical experiments}\label{section 5}

This section provides several numerical experiments illustrating the performance of the LSM and the ESM in solving the two-dimensional inverse biharmonic scattering problem for a cavity embedded in a thin elastic plate. In our simulations, the boundary of the model cavity is described parametrically as
\begin{align*}
    x(t)=(x_1(t),x_2(t))^\top,\quad 0\leq t<2\pi.
\end{align*}
The synthetic far-field data are generated by solving the corresponding direct scattering problems using the double-single layer potential boundary integral equation method introduced in \cite{DongHeping2024ANBI}. The exact parametric representations of the cavity boundaries are listed in Table~\ref{Table_1} and illustrated in Figure~\ref{fig:model_shapes}.

\begin{figure}[h]
	\centering	
	\subfigure[Apple-shaped cavity]{\includegraphics[width=0.3\textwidth]{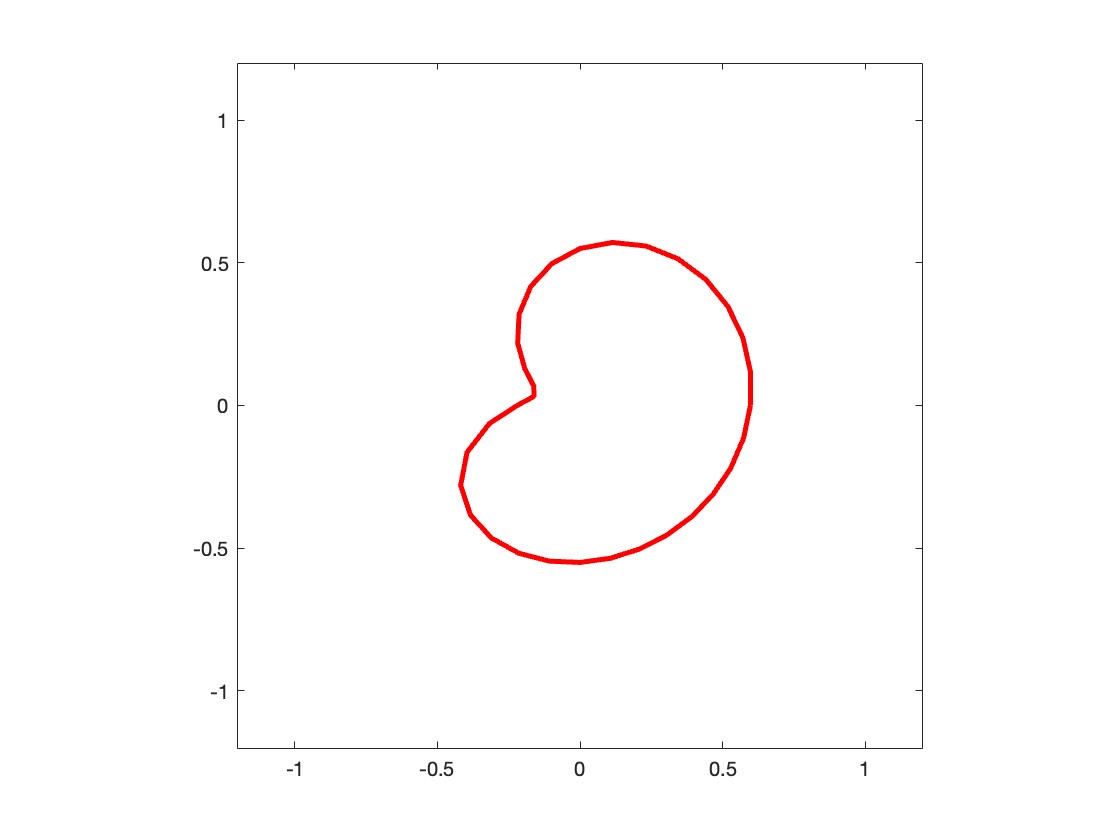}}\quad
	\subfigure[Peanut-shaped cavity]{\includegraphics[width=0.3\textwidth]{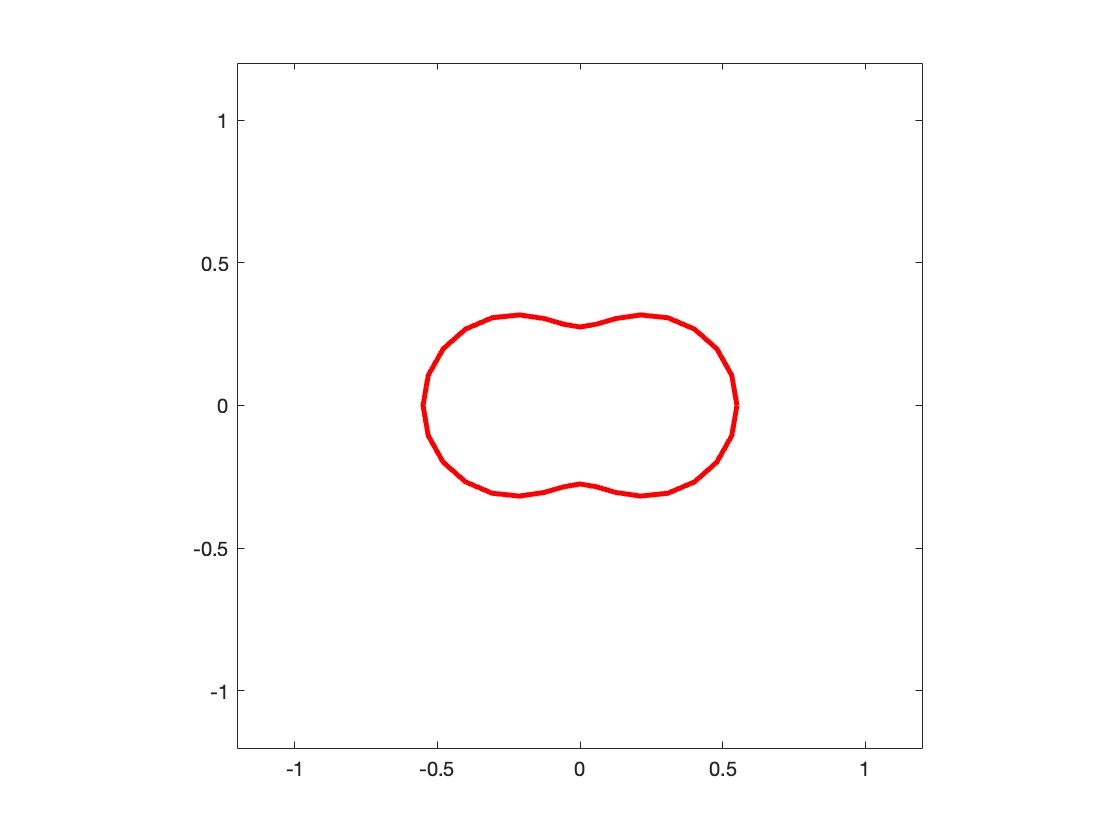}}
    \quad
	\subfigure[Peach-shaped cavity]{\includegraphics[width=0.3\textwidth]{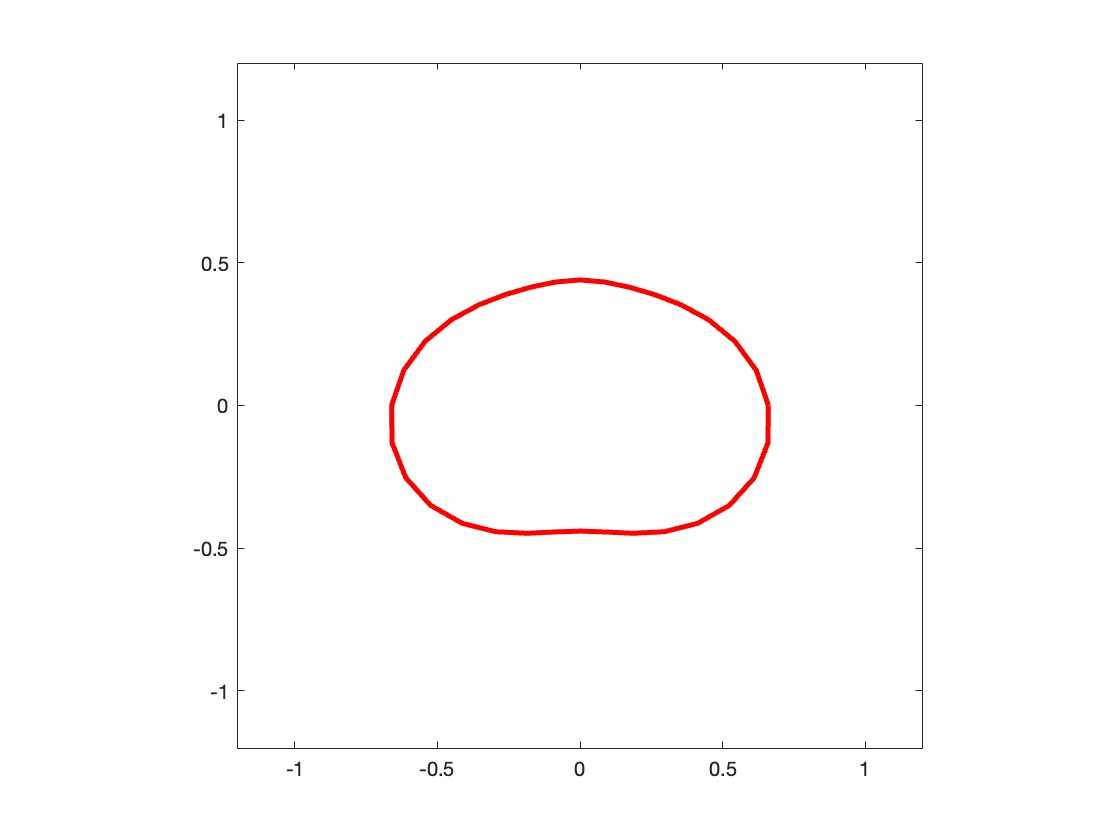}}
	\caption{Model cavities}\label{fig:model_shapes}
\end{figure}

\begin{table}[h!]
	\caption{The parameterized boundary curves.}
	\begin{tabular}{lll}
		\toprule[1pt]
		Boundary type           & Parameterization\\
		\midrule  
		Apple-shaped   & $x(t)=\displaystyle\frac{0.55(1+0.9\cos{t}+0.1\sin{2t})}{1+0.75\cos{t}}(\cos{t}, \sin{t}), \quad t\in [0,2\pi]$  
		\vspace{1.5ex}\\ 
		Peanut-shaped  & 
		$x(t)=0.275\sqrt{3\cos^2{t}+1}(\cos{t},\sin{t}), \quad
		t\in[0,2\pi]$
		\vspace{1.5ex} \\ 
		Peach-shaped  & 
		$x(t)=0.22(\cos^2{t}\sqrt{1-\sin{t}}+2)(\cos{t}, \sin{t}), \quad t\in[0,2\pi]$
		\vspace{1.5ex}\\
		\bottomrule[1pt]
	\end{tabular}\label{Table_1}
\end{table}

\subsection{The linear sampling method}

We use the system of boundary integral equations developed in \cite{DongHeping2024ANBI} to approximate the discretized far-field operator:
\begin{align*}
    \mathbf F=\left[ u^{\infty}(\Hat{x}_i,d_j) \right]_{i,j=1}^N, \quad \Hat{x}_i,d_j\in \mathbb S^1, 
\end{align*}
where $\mathbf F$ is an $N\times N$ complex-valued matrix corresponding to $N$ incident and observation directions. The directions are given by
\begin{align*}
\Hat{x}_i=d_i=(\cos (\theta_i), \sin (\theta_i))^\top, \quad\theta_i=2\pi (i-1)/N,\quad i=1,\dots,N.
\end{align*}
An additional quantity required is the far-field data vector of the fundamental solution, denoted by $\mathbf \varphi_z=\Phi(\cdot,z)$, which is computed as
\begin{align*}
    \mathbf \varphi_z=(\mathrm{e}^{-{\rm i}\kappa \Hat{x}_1\cdot z},\cdots, \mathrm{e}^{-{\rm i}\kappa \Hat{x}_{N}\cdot z})^\top,\quad z\in \mathbb R^2.
\end{align*}

To evaluate the stability of the method, we simulate experimental errors by adding random noise to the discretized far-field operator $\mathbf F$, resulting in the perturbed data:
\begin{align*}
    \mathbf F^{\delta}&=\left[\mathbf F_{i,j}(1+\delta \mathbf E_{i,j})\right]_{i,j=1}^{N},\quad \text{where }\norm{\mathbf E}_2=1.
\end{align*}
Here, $\mathbf E\in \mathbb C^{N\times N}$ is a random matrix with complex-valued entries, and $\delta>0$ denotes the relative noise level. In our numerical experiments, we consider noise levels of $\delta=2\%$ and $5\%$.

We numerically approximate the indicator function by solving for $\mathbf{g}_z^{\alpha}$ and plotting it's norm for any grid point $z$. Therefore, we that the indicator function for the LSM is given by 
\begin{align*}
    \mathcal I(z)=\frac{1}{\norm{\mathbf g_z^{\alpha}}_{\mathbb C^N}^2},
\end{align*}
where $\mathbf g_z^{\alpha}$ is the Tikhonov regularized solution to discretized far-field satisfying
\begin{align*}
(\alpha\mathbf I+\mathbf F^{*}\mathbf F)\mathbf g_z^{\alpha}=\mathbf F^{*}\mathbf \varphi_z,
\end{align*}
and $\norm{\cdot}_{\mathbb C^N}$ denotes the standard Euclidean norm on $\mathbb C^N$. We pick the regularization parameter ad-hoc such that $\alpha=10^{-6}$ to reconstruct the clamped cavities. In general, one can use a discrepancy principle to pick an optimal $\alpha$ but in our numerical experiments we see that this choice gives good reconstructions. Unless otherwise specified, the reconstructions are performed using far-field data with $N=32$ observation and incident directions.

In each example, the imaging domain is taken to be $[-1.5, 1.5]\times [-1.5, 1.5]$, discretized into a $128\times 128$ uniformly spaced grid. The boundary of the exact cavity is depicted by white dashed lines in the figures.

\subsubsection{Example 1. An apple-shaped cavity}

For the apple-shaped cavity, we consider wavenumbers $\kappa=\pi$ and $2\pi$. In the initial reconstructions shown in Figure \ref{fig:apple1}, we assume access to the discretized far-field operator with $N=32$ observation and incident directions, and no noise is added to the data. The results demonstrate that the spatial resolution of the reconstructed cavity improves with increasing wavenumber $\kappa$.

Figure \ref{fig:apple2} presents additional reconstructions of the apple-shaped cavity with random noise levels $\delta=0.02$ and $0.05$, corresponding to 2\% and 5\% relative noise, respectively. The reconstructions remain stable under noise, illustrating the robustness and effectiveness of the LSM in recovering the cavity shape. Figure \ref{fig:apple3} shows that the reconstruction becomes more robust to noise as the number of incident and observation directions increases.

\begin{figure}[htp]
\centering	
\subfigure[$\kappa=\pi$]{\includegraphics[width=0.48\textwidth]{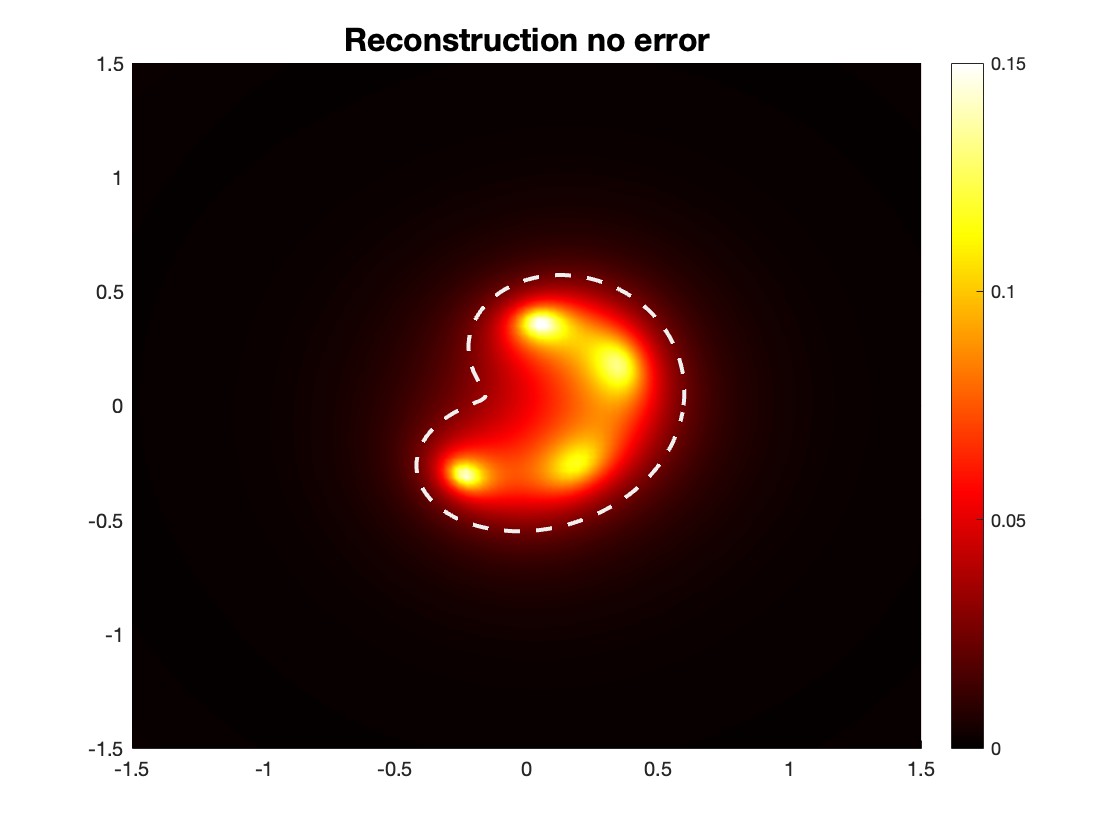}}\quad
\subfigure[$\kappa=2\pi$]{\includegraphics[width=0.48\textwidth]{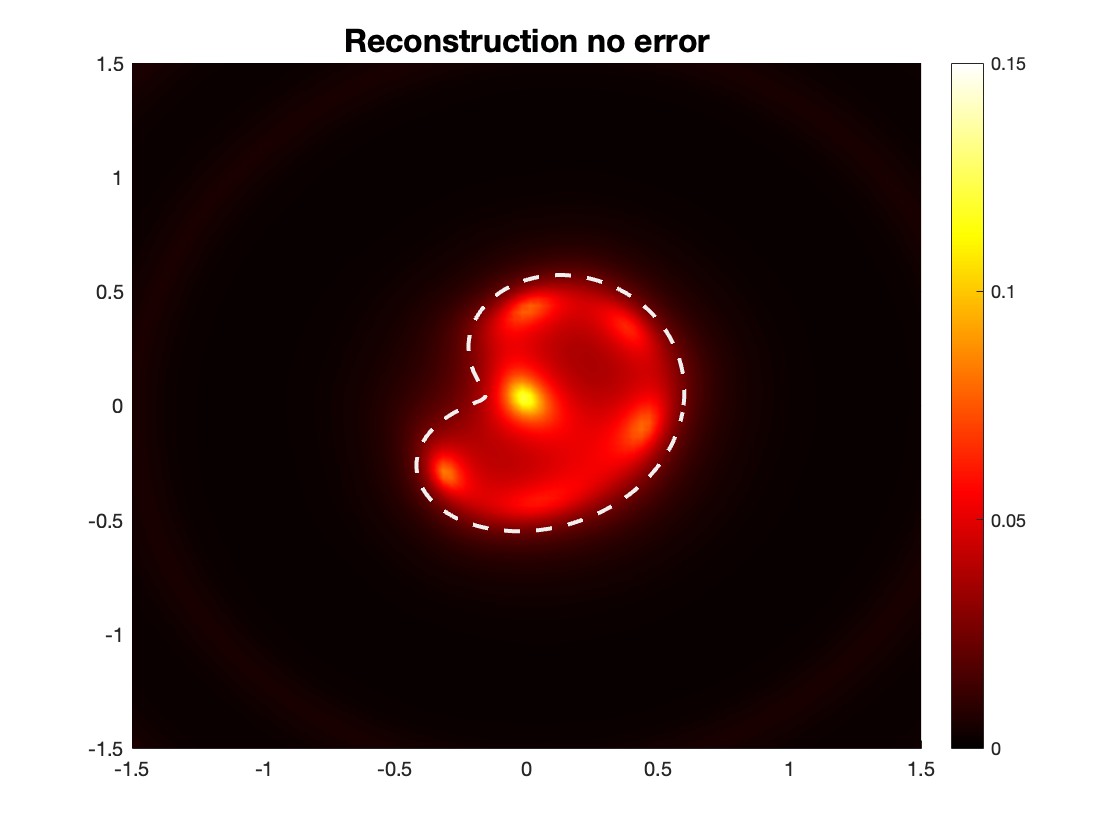}}
\caption{Example 1:  Reconstruction of the apple-shaped cavity using the LSM with (a) $\kappa=\pi$ and (b) $\kappa=2\pi$, employing a regularization parameter $\alpha=10^{-6}$. No random noise is added in this example.}\label{fig:apple1}
\end{figure}

\begin{figure}[htp]
\centering	
\subfigure[$\delta=2\%$]{\includegraphics[width=0.48\textwidth]{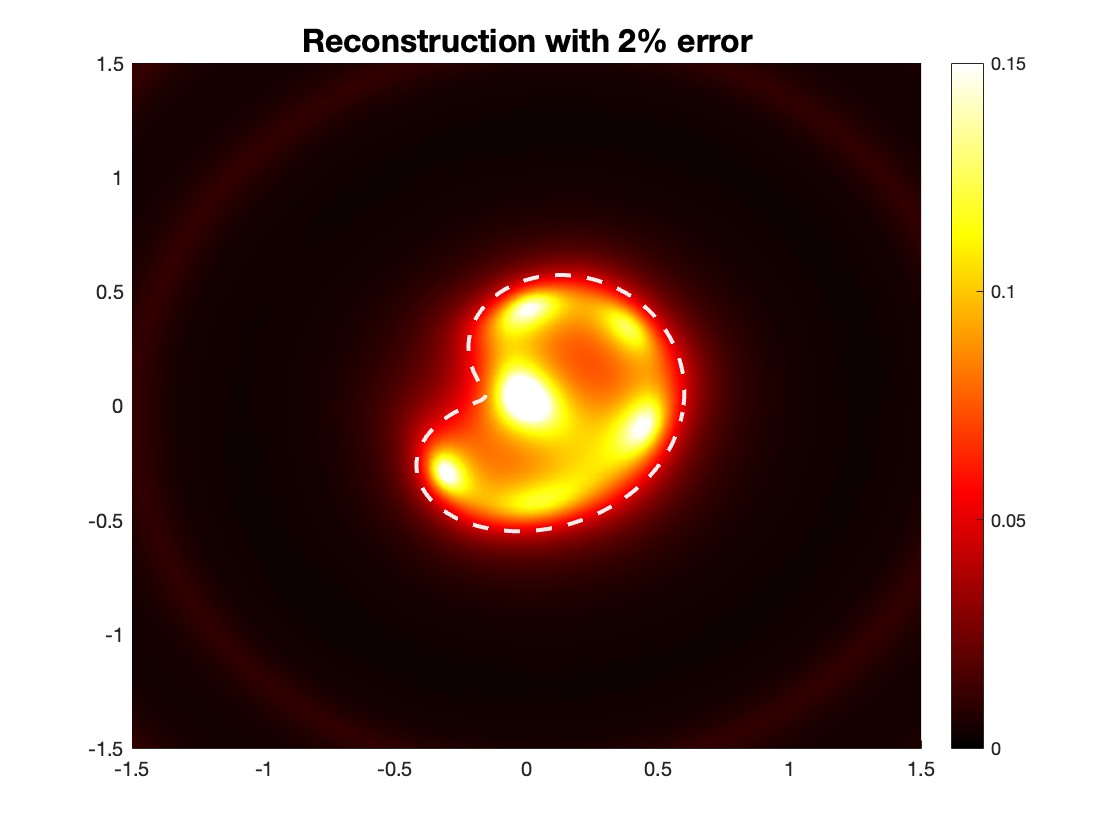}}\quad
\subfigure[$\delta=5\%$]{\includegraphics[width=0.48\textwidth]{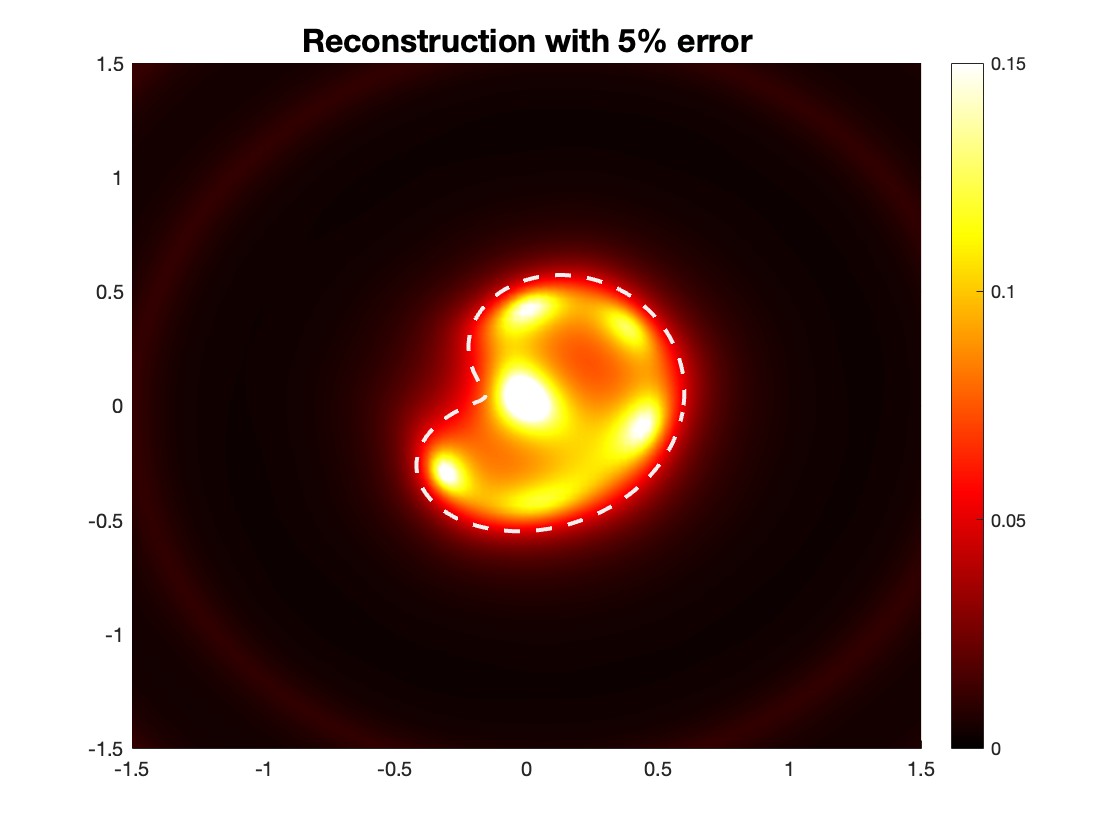}}
\caption{Example 1: Reconstruction of the apple-shaped cavity using the LSM at wavenumber $\kappa=2\pi$, with noise levels (a) $\delta=2\%$ and (b) $\delta=5\%$. The regularization parameter is set to $\alpha=10^{-6}$.}\label{fig:apple2}
\end{figure}

\begin{figure}[htp]
\centering	
\subfigure[$N=64$]{\includegraphics[width=0.48\textwidth]{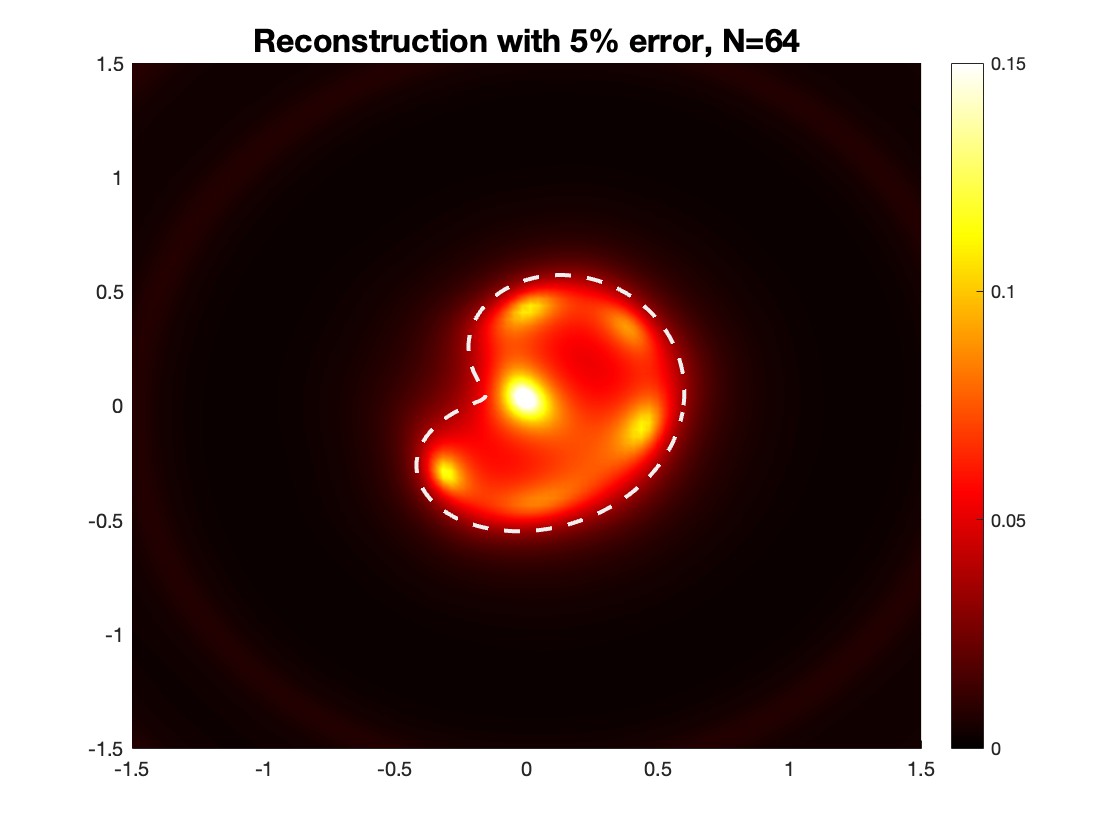}}\quad
\subfigure[$N=128$]{\includegraphics[width=0.48\textwidth]{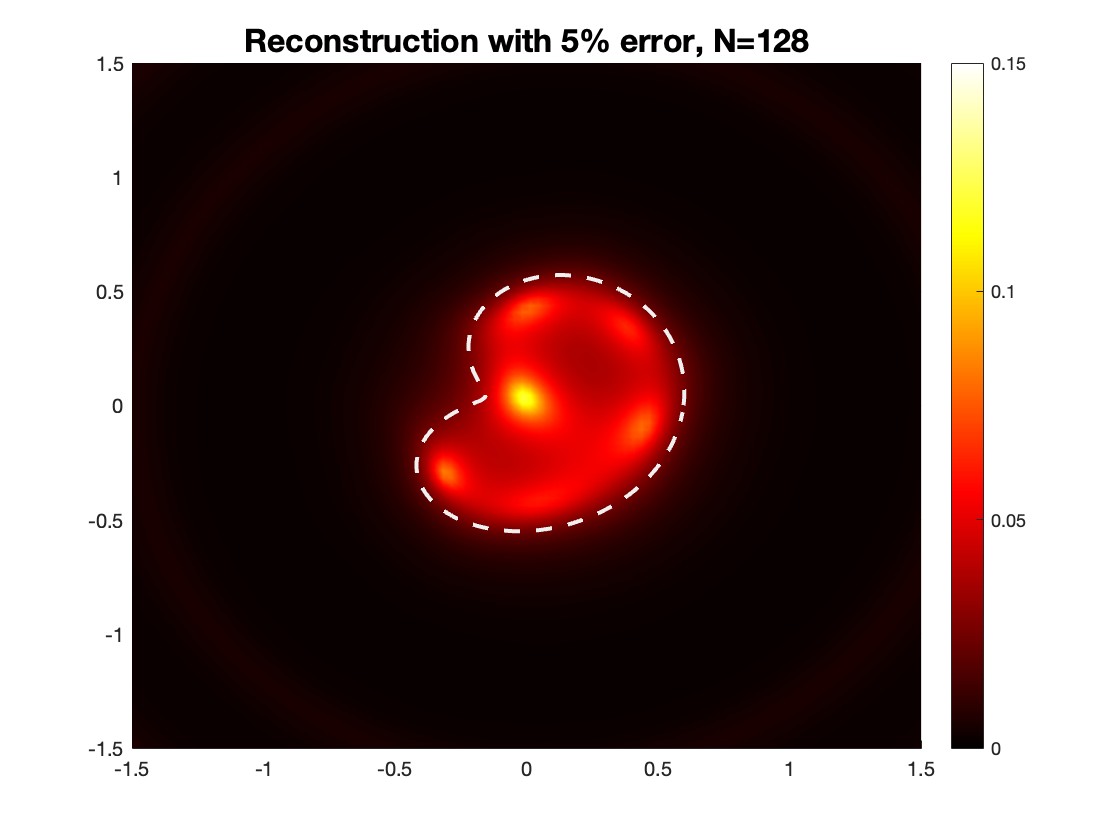}}
\caption{Example 1: Reconstruction of the apple-shaped cavity using the LSM at wavenumber $\kappa=2\pi$, with $5\%$ random noise added to the data. The reconstructions are performed with varying numbers of incident and observation directions: (a) $N=64$ and (b) $N=128$.}\label{fig:apple3}
\end{figure}

\subsubsection{Example 2. A peanut-shaped cavity}

For this reconstruction, we use the same physical parameters as in the apple-shaped cavity. In Figure~\ref{fig:peanut1}, we observe that the reconstructions of the peanut-shaped cavity are again reasonably accurate, with improved spatial resolution as the wavenumber $\kappa$ increases. Using the LSM imaging function, we successfully recover the cavity's location, size, and shape.

Figure~\ref{fig:peanut2} demonstrates that the reconstructions remain robust in the presence of random noise, highlighting the effectiveness of the LSM in identifying the clamped cavity. As shown in Figure~\ref{fig:peanut3}, the reconstructions exhibit increased robustness to noise when the number of incident and observation directions satisfies $N\geq 64$.

\begin{figure}[htp]
\centering	
\subfigure[$\kappa=\pi$]{\includegraphics[width=0.48\textwidth]{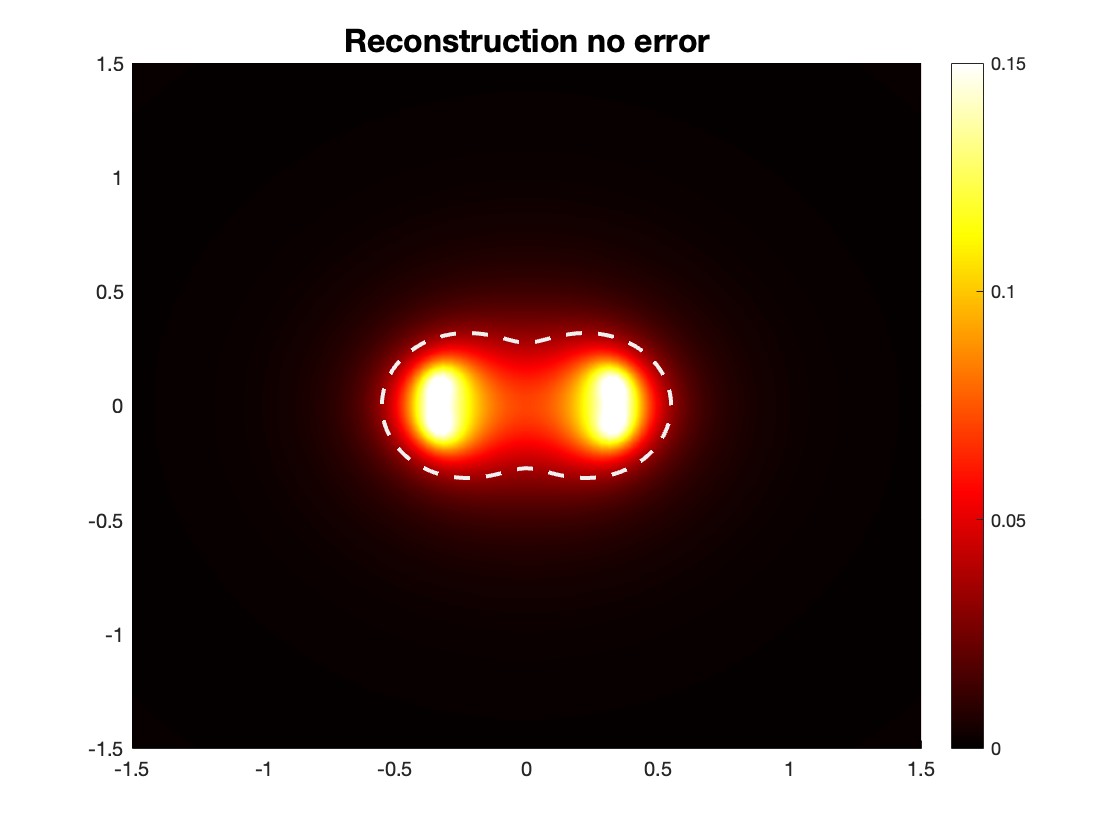}}\quad
\subfigure[$\kappa=2\pi$]{\includegraphics[width=0.48\textwidth]{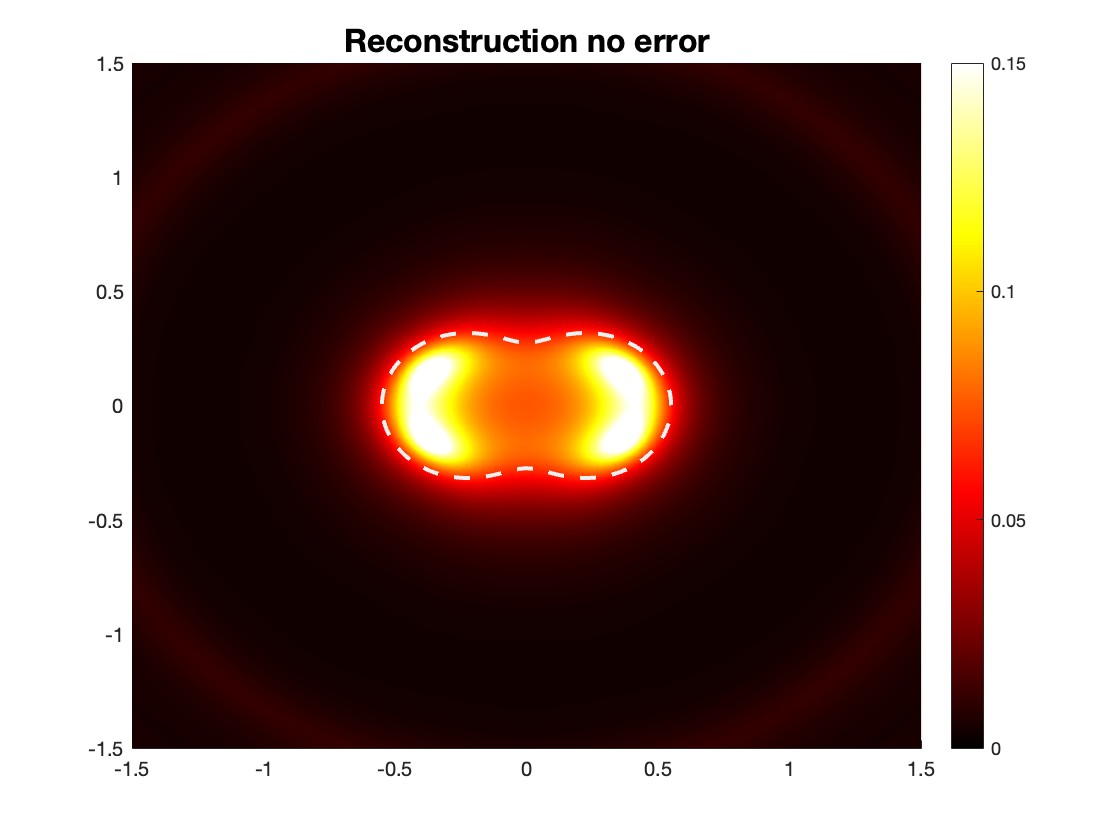}}
\caption{Example 2: Reconstruction of the peanut-shaped cavity using the LSM with (a) $\kappa=\pi$ and (b) $\kappa=2\pi$. No random noise is added in these experiments.}\label{fig:peanut1}
\end{figure}

\begin{figure}[htp]
\centering
\subfigure[$\delta=2\%$]{\includegraphics[width=0.48\textwidth]{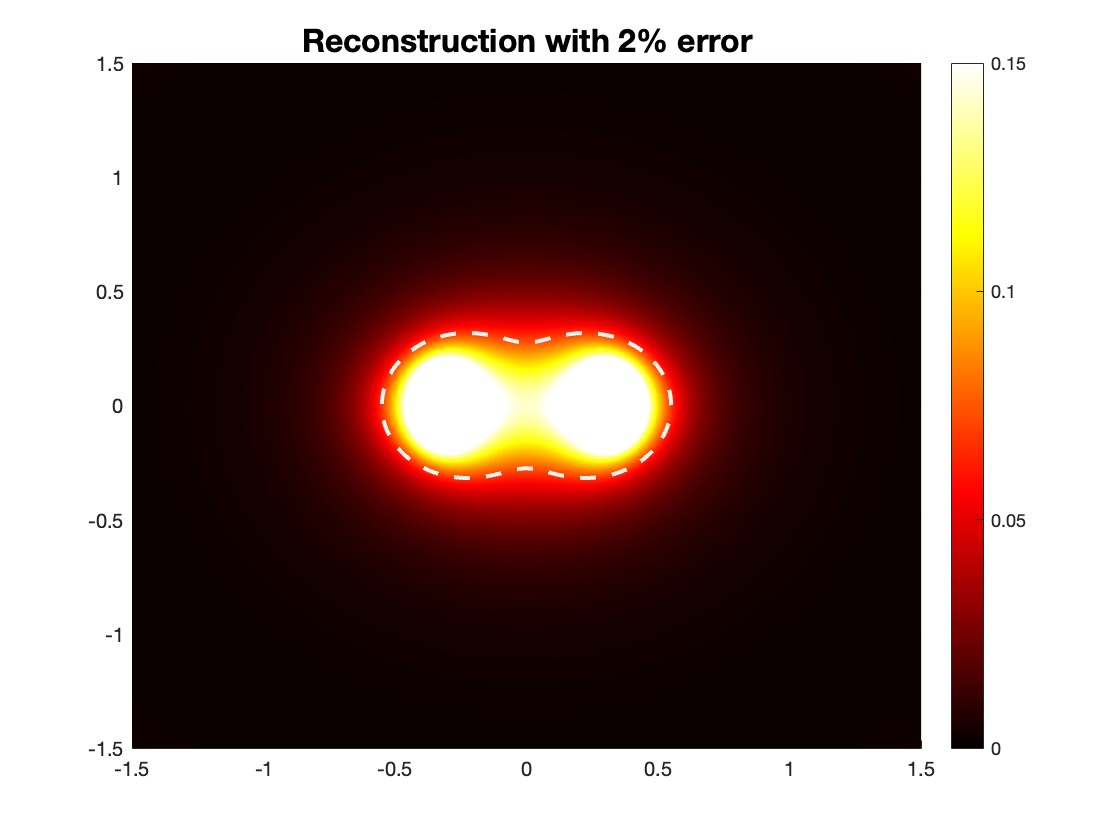}}\quad
\subfigure[$\delta=5\%$]{\includegraphics[width=0.48\textwidth]{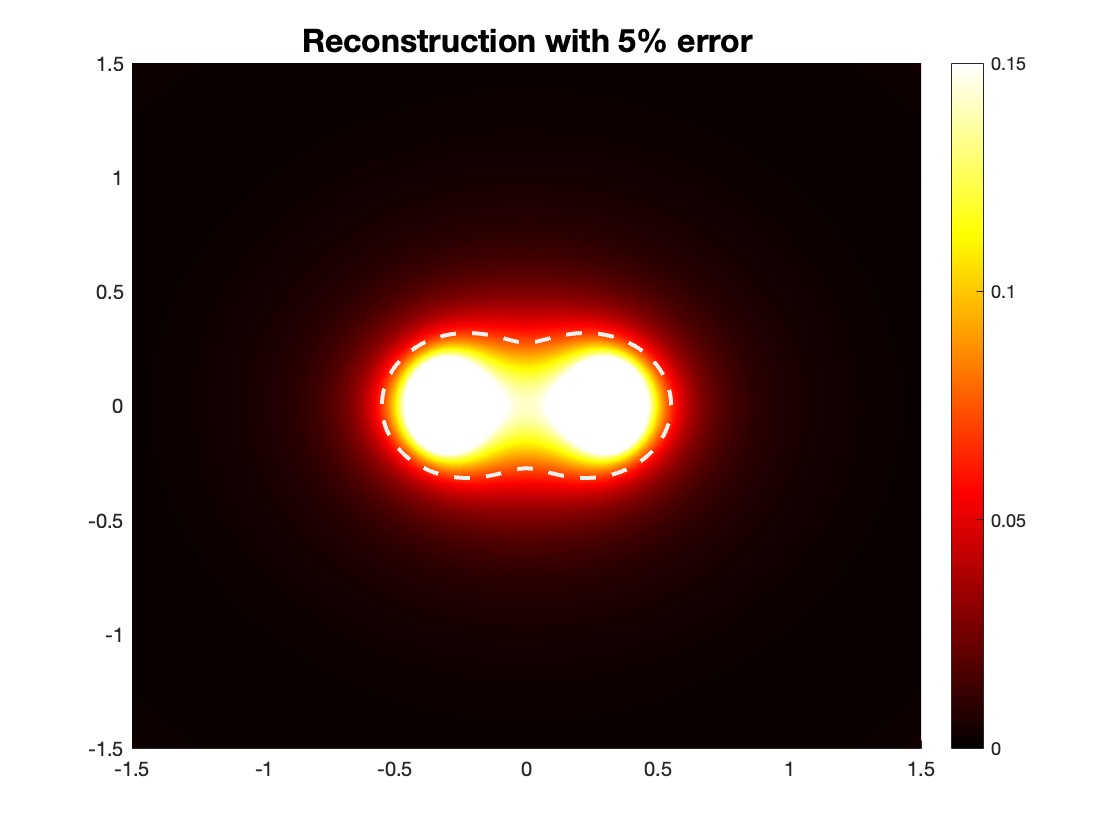}}
\caption{Example 2: Reconstruction of the peanut-shaped cavity using the LSM with a wavenumber of $\kappa=\pi$ and added random noise. The noise levels are (a) $\delta=2\%$ and (b) $\delta=5\%$, with a regularization parameter $\alpha=10^{-6}$.}\label{fig:peanut2}
\end{figure}

\begin{figure}[htp]
\centering	
\subfigure[$N=64$]{\includegraphics[width=0.48\textwidth]{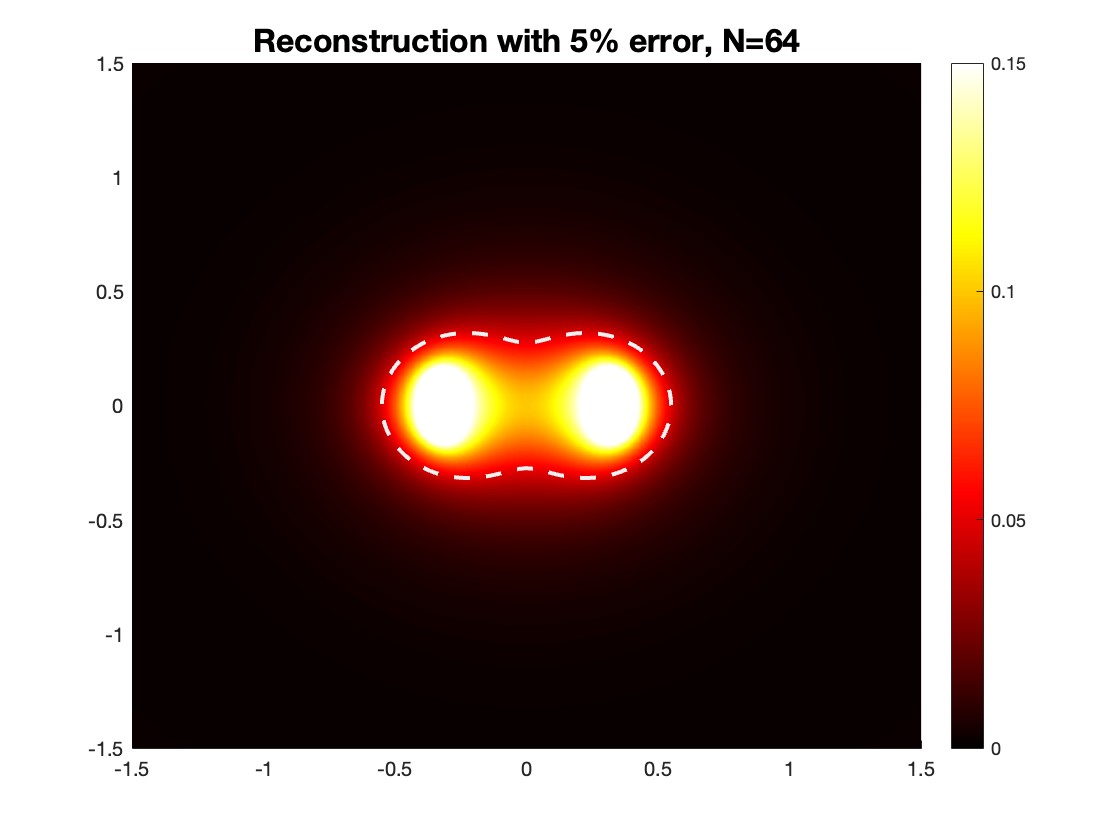}}\quad
\subfigure[$N=128$]{\includegraphics[width=0.48\textwidth]{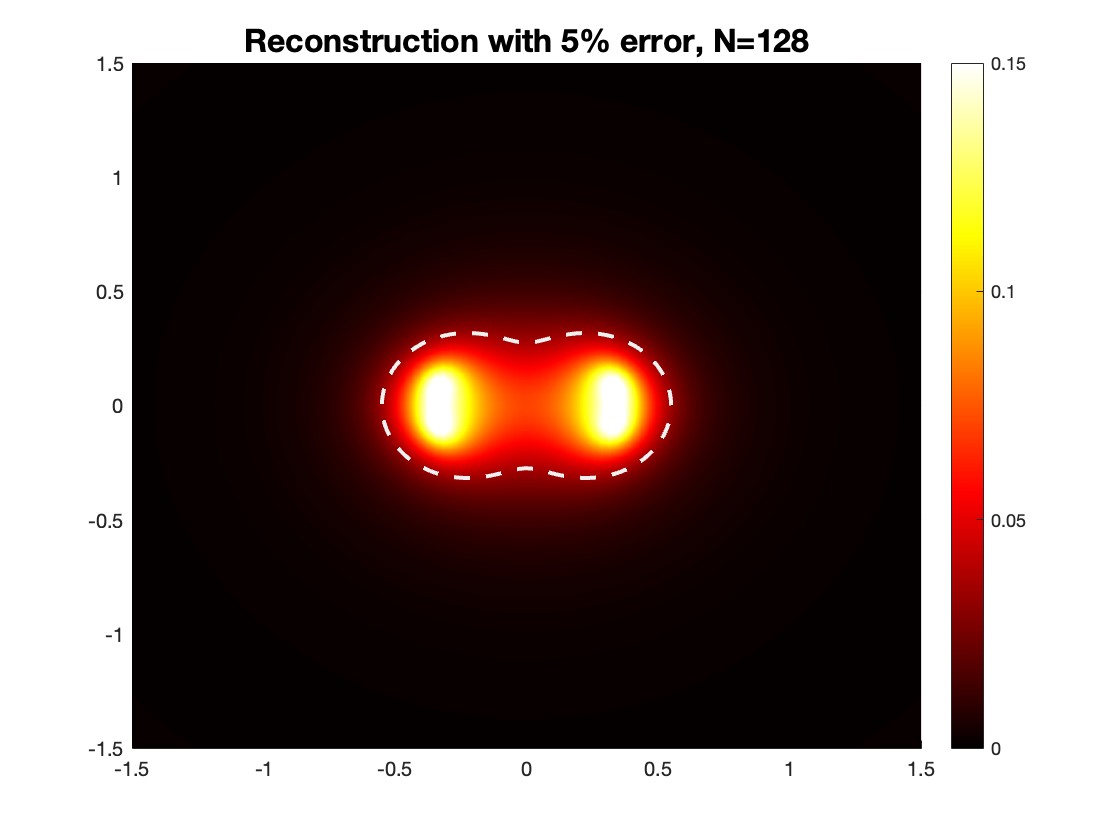}}
\caption{Example 2: Reconstruction of the peanut-shaped cavity using the LSM with a wavenumber of $\kappa=\pi$ and $5\%$ random noise added. The reconstructions are shown for varying numbers of incident and observation directions: (a) $N=64$ and (b) $N=128$.}\label{fig:peanut3}
\end{figure}

\subsubsection{Example 3. A peach-shaped cavity}

Unlike the apple- and peanut-shaped cavities, the peach-shaped cavity does not possess an analytic boundary; its first derivative exhibits a singularity at $t=\pi/2$. In Figure \ref{fig:peach1}, the reconstructions are reasonably accurate, with improved spatial resolution observed as the wavenumber $\kappa$ increases. These results demonstrate the effectiveness of the LSM in accurately reconstructing cavities with non-analytic boundaries. Figure~\ref{fig:peach2} further illustrates the robustness of the LSM in the presence of random noise for the peach-shaped cavity. As shown in Figure~\ref{fig:peach3}, the reconstructions become more robust to noise with an increasing number of incident and observation directions.

\begin{figure}[htp]
\centering	
\subfigure[$\kappa=\pi$]{\includegraphics[width=0.48\textwidth]{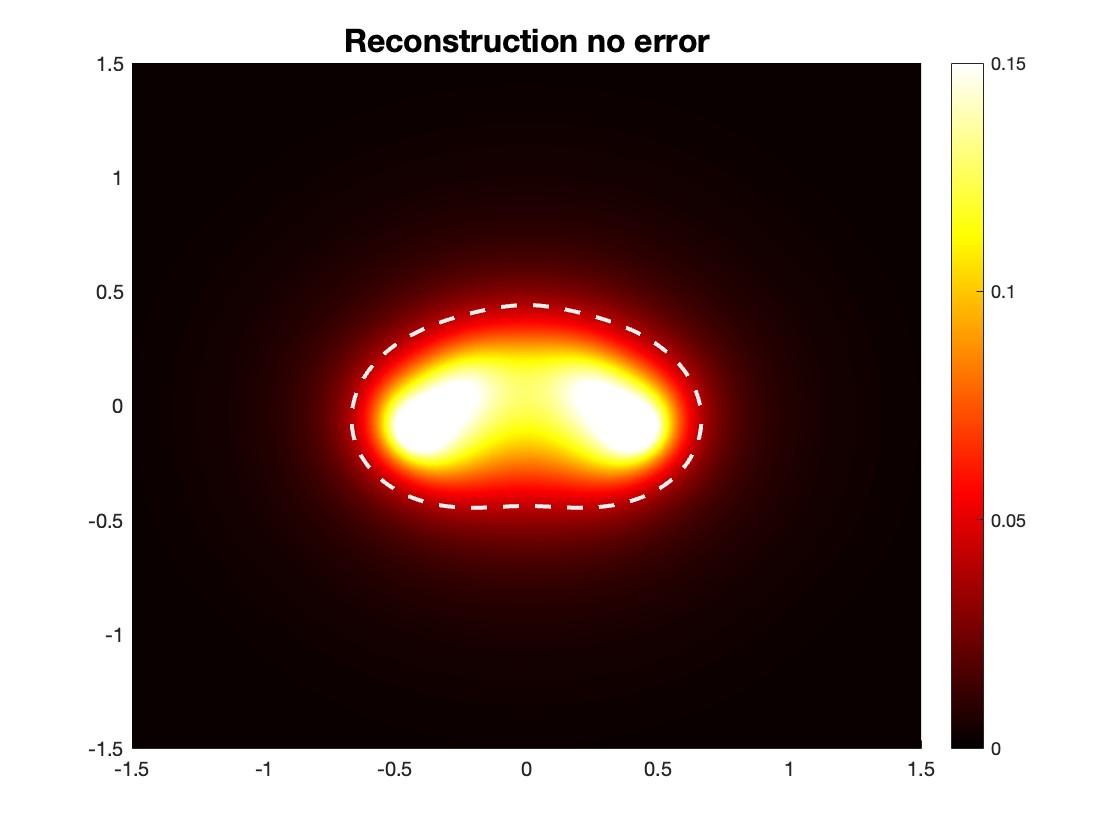}}\quad
\subfigure[$\kappa=2\pi$]{\includegraphics[width=0.48\textwidth]{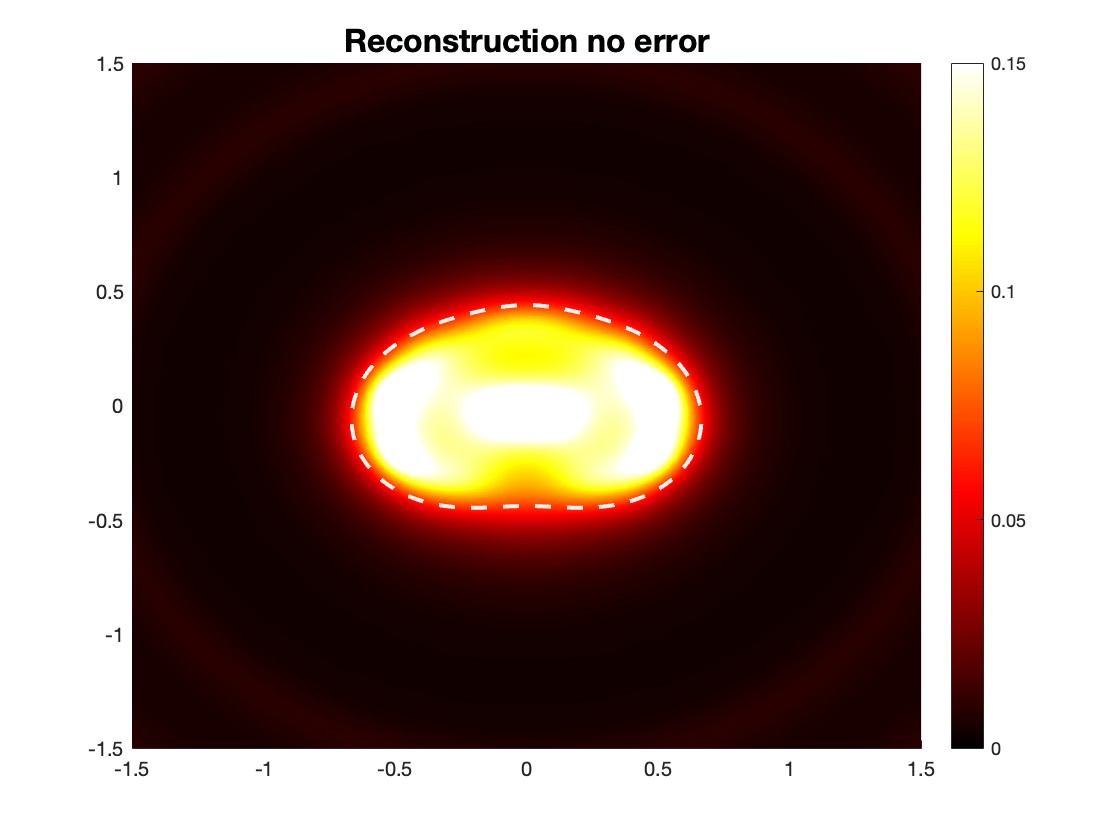}}
\caption{Example 3: Reconstruction of the peach-shaped cavity using the LSM with (a) $\kappa = \pi$ and (b) $\kappa = 2\pi$, using a regularization parameter of $\alpha = 10^{-6}$.}\label{fig:peach1}
\end{figure}

\begin{figure}[htp]
\centering	
\subfigure[$\delta=2\%$]{\includegraphics[width=0.48\textwidth]{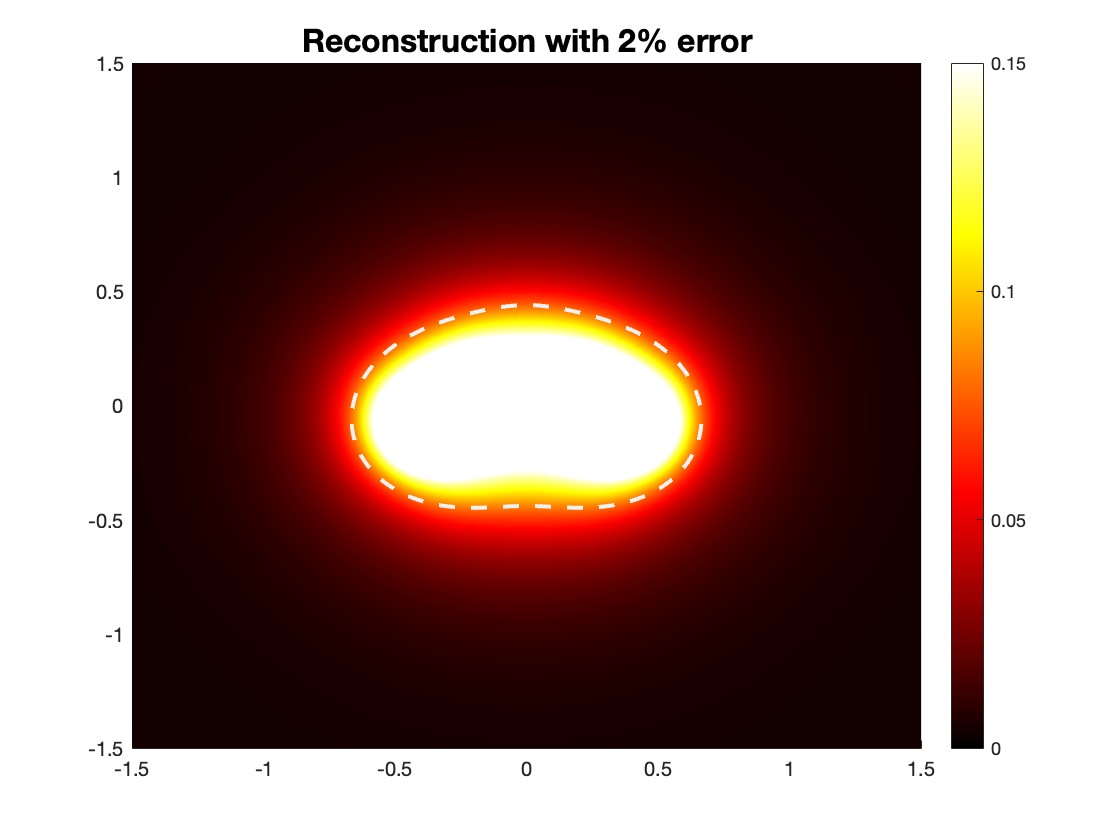}}\quad
\subfigure[$\delta=5\%$]{\includegraphics[width=0.48\textwidth]{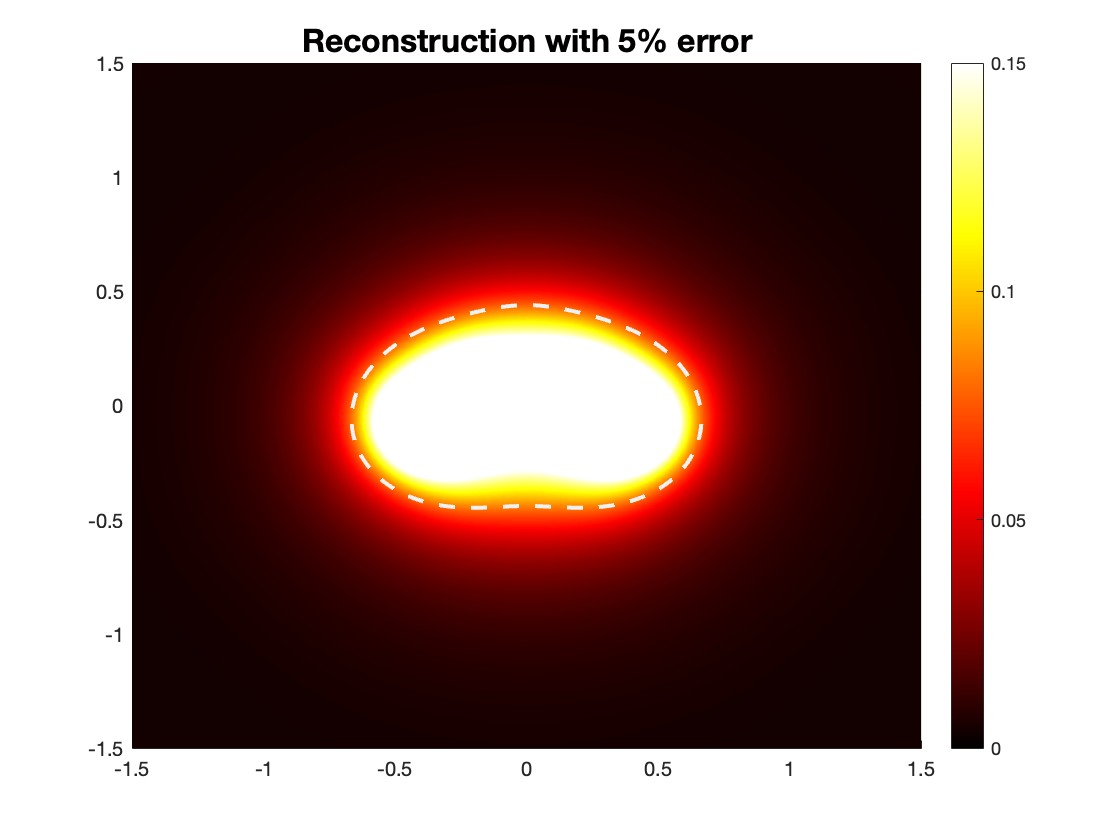}}
\caption{Example 3: Reconstruction of the peach-shaped cavity using the LSM with a wavenumber of $\kappa = \pi$ and added random noise. The noise levels are (a) $\delta = 2\%$ and (b) $\delta = 5\%$, with a regularization parameter of $\alpha = 10^{-6}$.}\label{fig:peach2}
\end{figure}

\begin{figure}[htp]
\centering	
\subfigure[$N=64$]{\includegraphics[width=0.48\textwidth]{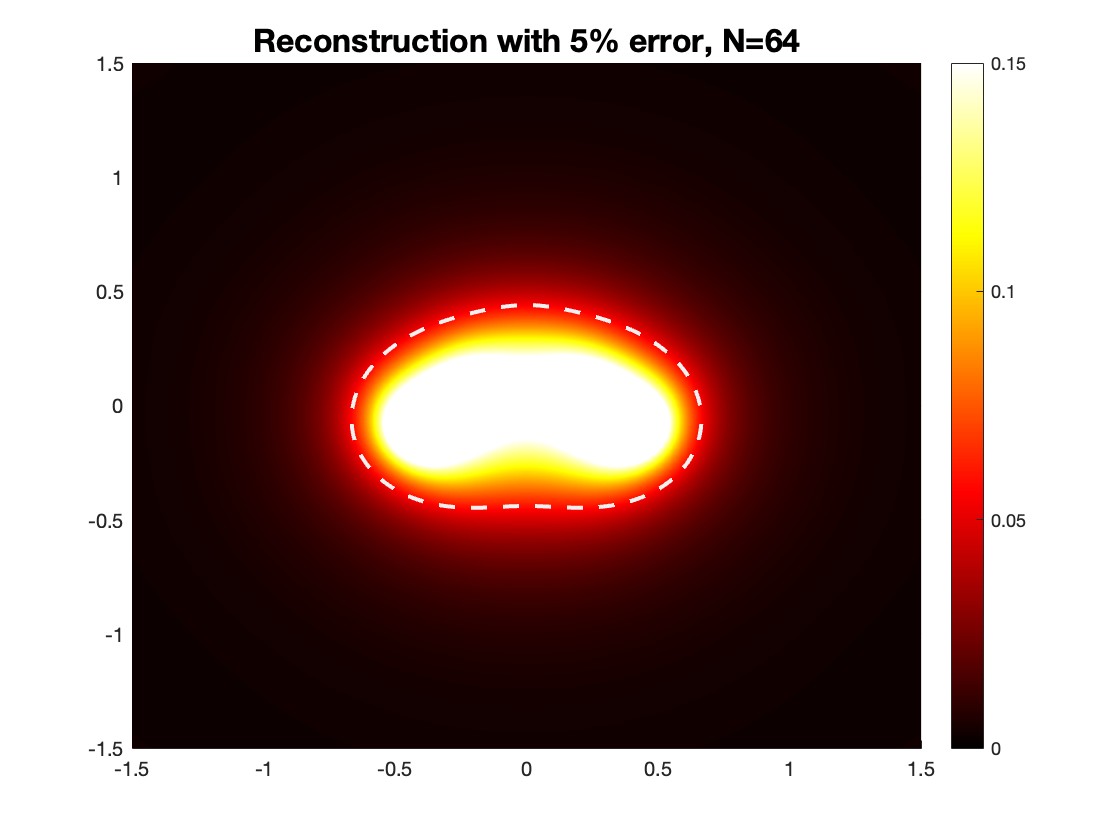}}\quad
\subfigure[$N=128$]{\includegraphics[width=0.48\textwidth]{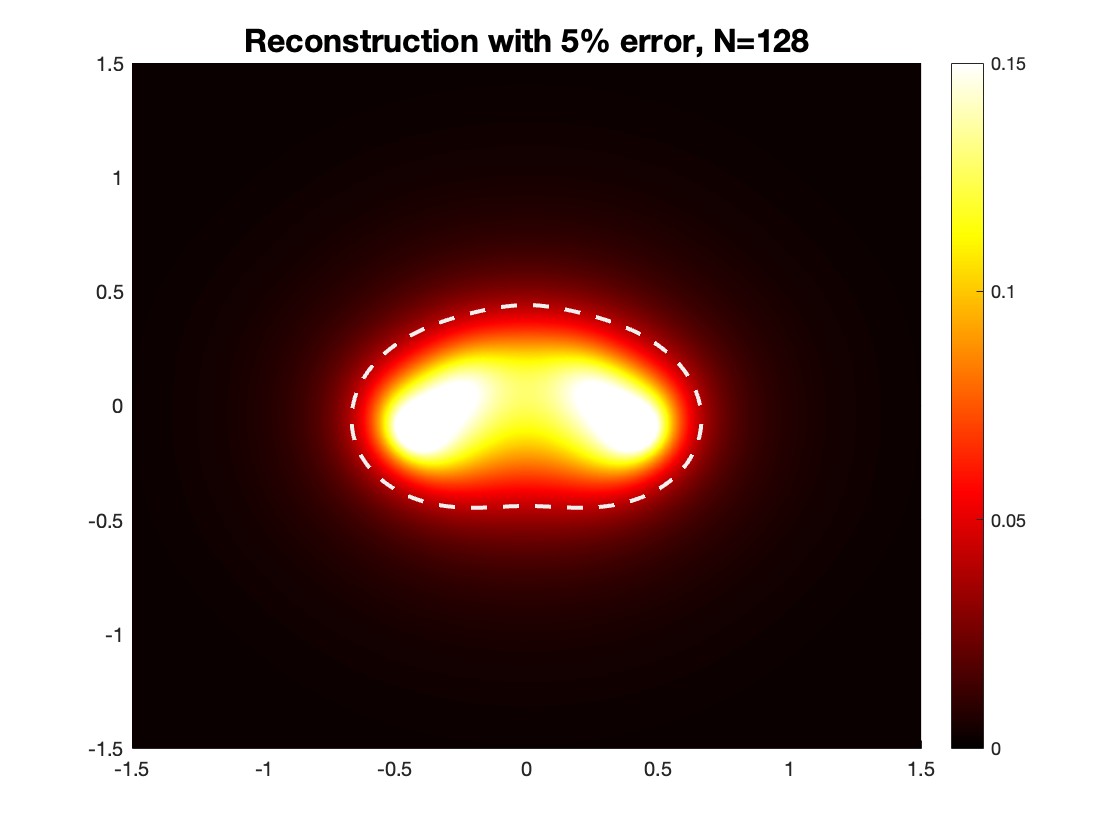}}
\caption{Example 3: Reconstruction of the peach-shaped cavity using the LSM with a wavenumber of $\kappa = \pi$, $5\%$ random noise, and a regularization parameter of $\alpha = 10^{-6}$, for varying numbers of incident and observation directions: (a) $N = 64$ and (b) $N = 128$.}\label{fig:peach3}
 \end{figure}

\subsubsection{Recovering the unit ball with a Dirichlet eigenvalue}

In addition to assuming that the wavenumber $\kappa^2$ is not an eigenvalue of the clamped transmission problem given by (\ref{eigenvalue_problem}), the main result in \cite{guoetal2024} also assumes that $\kappa^2$ is not a Dirichlet eigenvalue of $-\Delta$ in $D$. However, this latter assumption is not essential for the effective reconstruction of a clamped cavity. For instance, consider the case where $\kappa^2 = \lambda$, with $\lambda$ being an eigenvalue of the Dirichlet problem:
\begin{align*}
-\Delta \phi = \lambda \phi \quad \text{in } B_1(0), \quad \phi = 0 \quad \text{on } \partial B_1(0).
\end{align*}
The eigenvalues of this problem are given by $\lambda_{mn} = j_{mn}^2$, where $j_{mn}$ denotes the $m$-th positive zero of the Bessel function $J_n(r)$ of order $n$. Thus, $\kappa_{mn} = \sqrt{\lambda_{mn}} = j_{mn}$. Figure~\ref{fig:unitball} illustrates the reconstruction of the unit disk $D = B_1(0)$ using $\kappa_1 \approx 2.40483$ and $\kappa_2 \approx 5.5201$, which are the first and second roots of $J_0(r)$ and correspond to the Dirichlet eigenvalues of $-\Delta$ in $B_1(0)$.

\begin{figure}[htp]
\centering	
\subfigure[$\kappa=j_{01}$]{\includegraphics[width=0.48\textwidth]{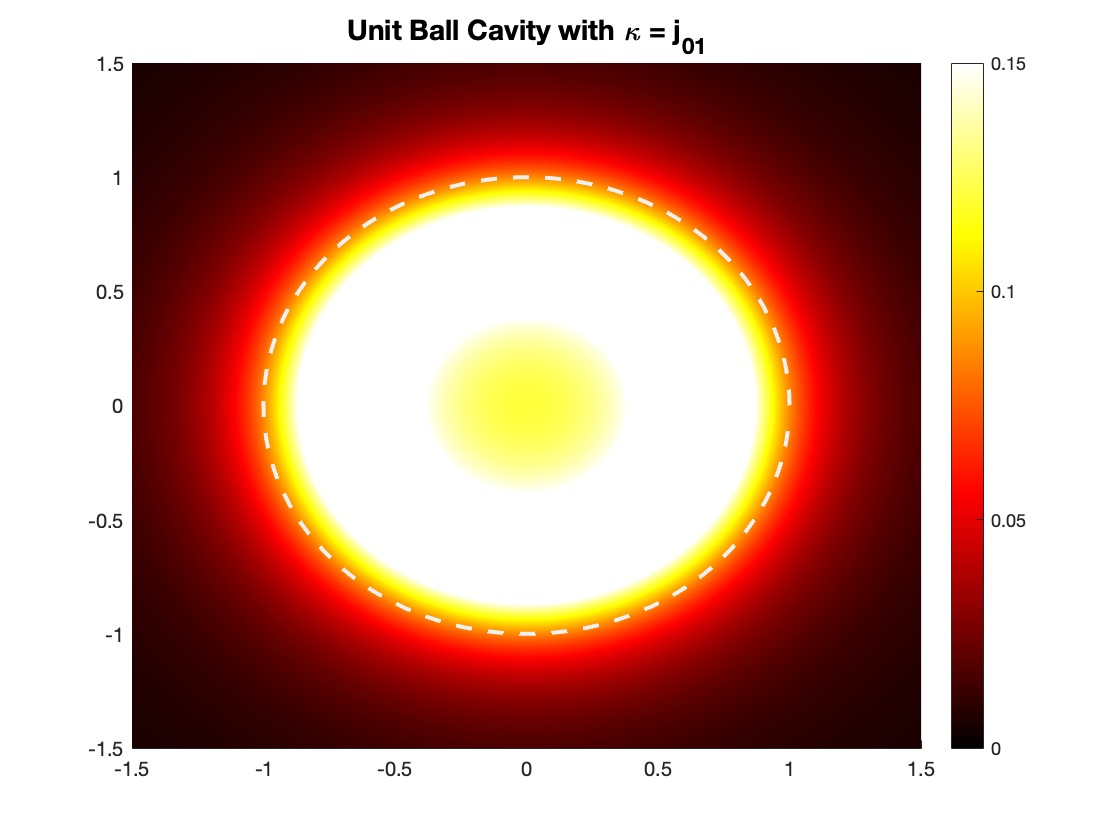}}\quad
\subfigure[$\kappa=j_{02}$]{\includegraphics[width=0.48\textwidth]{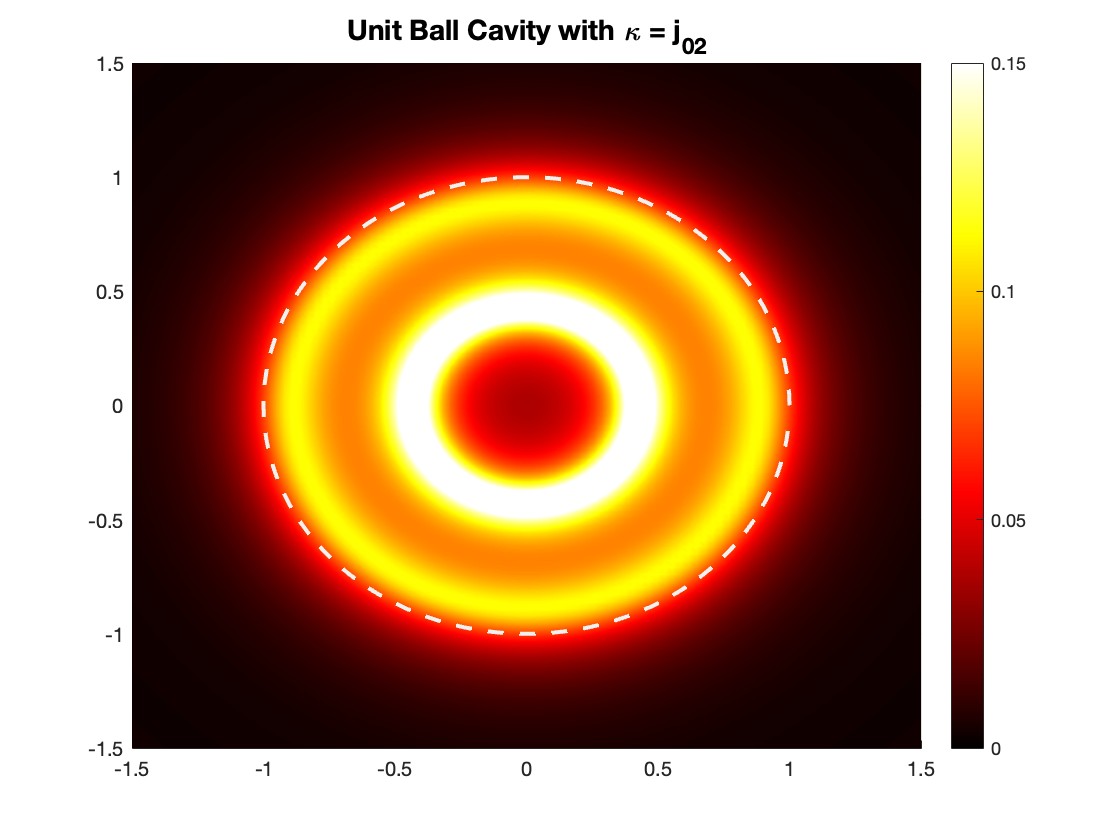}}
\caption{Example 4: Reconstruction of the unit disk using the LSM with no added noise and a regularization parameter of $\alpha = 10^{-6}$. (a) $\kappa = j_{01}$; (b) $\kappa = j_{02}$, where $j_{01}$ and $j_{02}$ are the first and second positive zeros of the Bessel function $J_0(r)$, respectively.}\label{fig:unitball}
 \end{figure}
 
\subsection{The extended sampling method}\label{section_ESM_numerics}
 
 We present several numerical examples to demonstrate the effectiveness of the ESM in recovering the location of clamped cavities. As test geometries, we again consider the apple-, peanut-, and peach-shaped cavities defined in Table \ref{Table_1}. The far-field data $u^{\infty}(\Hat x,d)$ , for $(\Hat{x},d)\in \mathbb S^1\times \mathbb{S}^1_{\text{inc}}$, are computed as discussed in the previous section. In each example, we employ an equally spaced $200\times 200$ sampling grid over the imaging domain. For each sampling point $z$, we apply Tikhonov regularization with a fixed parameter $\alpha=10^{-4}$ to solve the discretized modified far-field equation. This gives that (\ref{approx_ESM_FF_eqn}) becomes linear system
 \[
 \mathbf A^z\mathbf g_z^{\alpha}=u^{\infty}(\cdot \, ,d),
 \]
 where the matrix $\mathbf{A}^z$ is defined as
 \begin{align*}
\mathbf A_{i,j}^z=\mathrm{e}^{{\rm i}\kappa \cdot z\cdot(\hat{y}_j-\hat{x}_i)}U_B^{\infty}(\Hat{x}_i,\Hat{y}_j),\quad i,j=1,2,\dots,40.
 \end{align*}
 The regularized solution is computed by
 \begin{align*}
\mathbf g_z^{\alpha}(d)\approx ((\mathbf A^z)^{*}\mathbf A^z+\alpha \mathbf I)^{-1}(\mathbf A^z)^{*}u^{\infty}(\cdot \, ,d),
 \end{align*}
 where $\mathbf I$ is the identity matrix. The discrete indicator function is then defined as
 \begin{align*}
\mathcal I(z)= \frac{ \|\mathbf{g}_z^{\alpha} \|_{\ell^2}}{\max\limits_{z \in \mathcal{M}}  \|\mathbf{g}_z^{\alpha} \|_{\ell^2} }
 \end{align*}
 for all sampling points $z$ where we pick the location of the cavity to be the minimizer of $\mathcal{I}(z)$ on the sampling grid $\mathcal{M}$.  The discrete indicator functions for multiple incident directions and multiple frequencies are defined similarly using equations (\ref{ESM_ind_1}) and (\ref{ESM_ind_2}), respectively.
 
\subsubsection{A fixed incident direction} 
 
For our selected numerical examples, we consider a fixed incident direction given by
\begin{align*}
d_0=\left\{(\cos{\theta},\sin{\theta})\,|\,\theta=\pi/3 \right\}=\left\{ \left(\frac{1}{2},\frac{\sqrt 3}{2} \right) \right\},
\end{align*}
with a full observation aperture $\mathbb S^1$. 

We use $40$ observation directions in all the reconstructions. Figures \ref{fig:esmapple}, \ref{fig:esmpeanut}, and \ref{fig:esmpeach} present multilevel ESM reconstructions of the location of apple-, peanut-, and peach-shaped cavities, both centered at the origin and shifted to $(-1.5,1.5)$, using the configuration 
$\mathbb{S}^1 \times \{  d_0 \} $ i.e., a single incident direction. Since the size of the cavity is not known a priori, the multilevel ESM is employed to identify an appropriate radius $R$ for the sampling disks. Starting from an initial value of $R=4.0$, the radius is successively decreased until a satisfactory resolution is achieved. For cavities centered at the origin, the optimal sampling radius is determined to be $R=0.5$, while for cavities centered at $(-1.5,1.5)$, the optimal radius is $R=2.5$.

 \begin{figure}[htp]
\centering	
\subfigure[The apple-shaped cavity centered at the origin]{\includegraphics[width=0.48\textwidth]{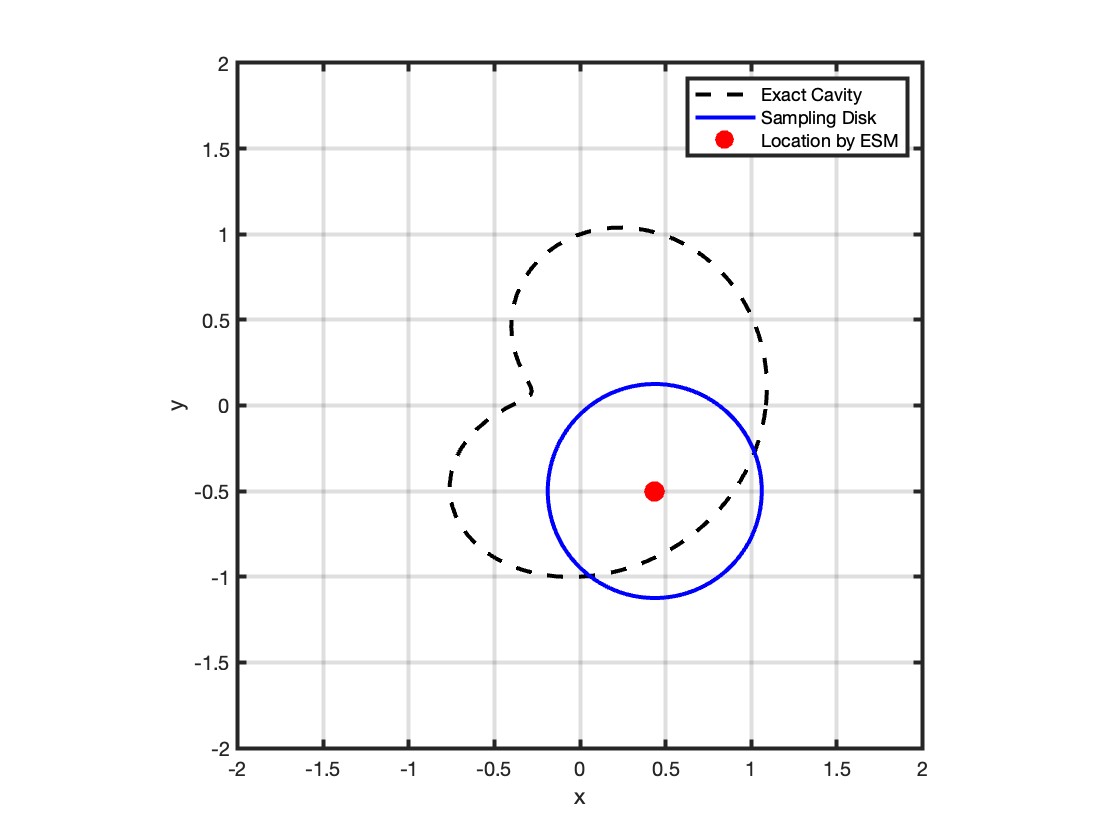}}\quad
\subfigure[The apple-shaped cavity shifted to $(-1.5,1.5)$]{\includegraphics[width=0.48\textwidth]{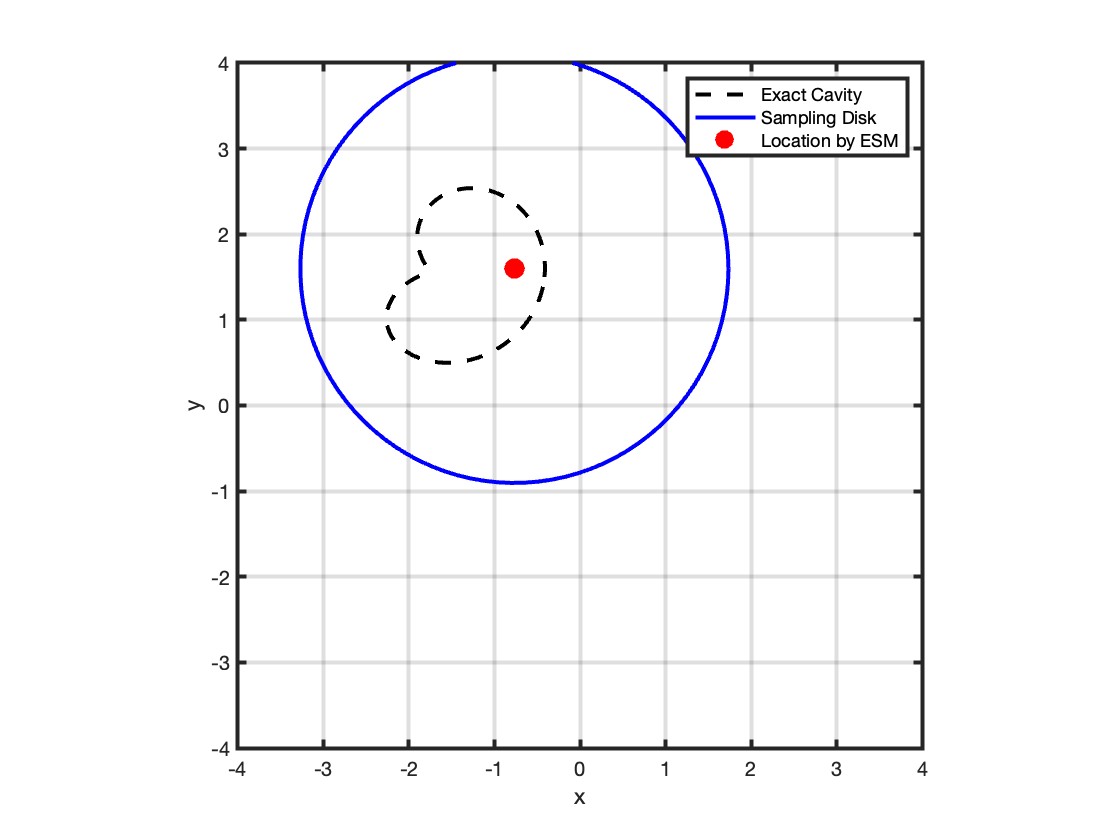}}
\caption{Reconstruction results using multilevel ESM for (a) the apple-shaped cavity centered at the origin and (b) the apple-shaped cavity shifted to $(-1.5,1.5)$, based on a single incident direction $d_0$.}\label{fig:esmapple}
\end{figure}
 
\begin{figure}[htp]
\centering	
\subfigure[The peanut-shaped cavity centered at the origin]{\includegraphics[width=0.48\textwidth]{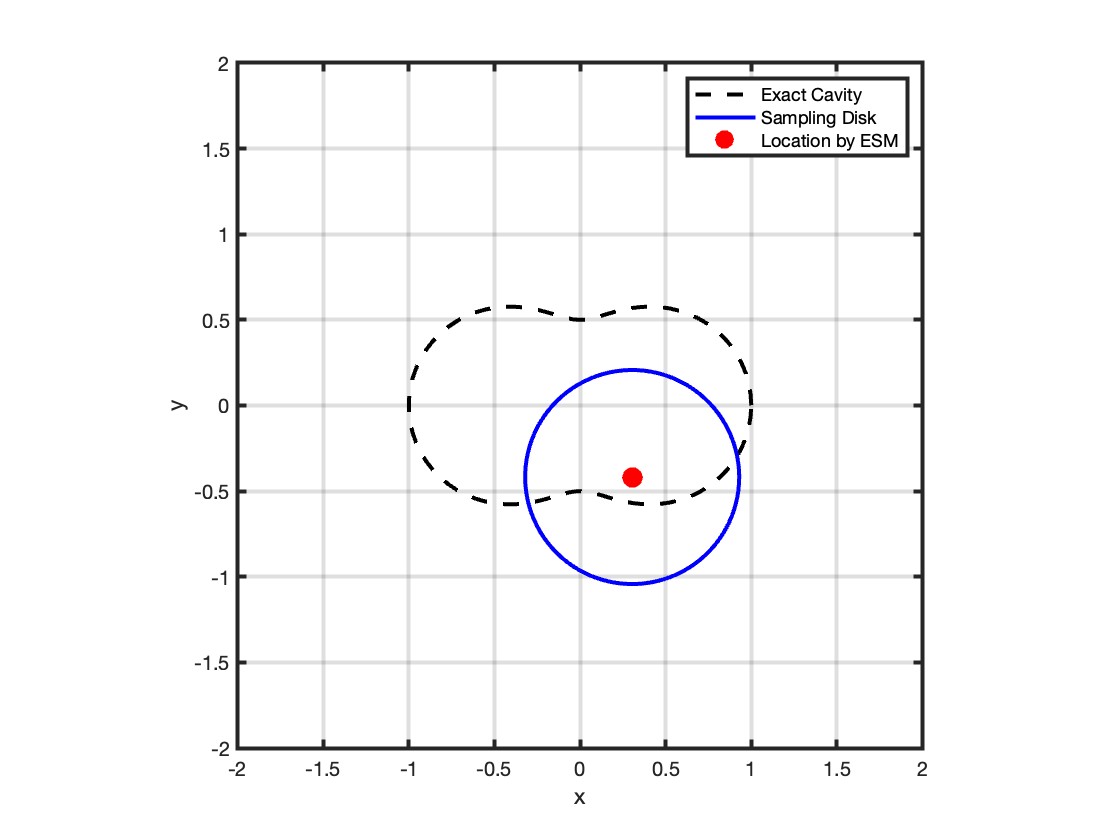}}\quad
\subfigure[The peanut-shaped cavity shifted to $(-1.5,1.5)$]{\includegraphics[width=0.48\textwidth]{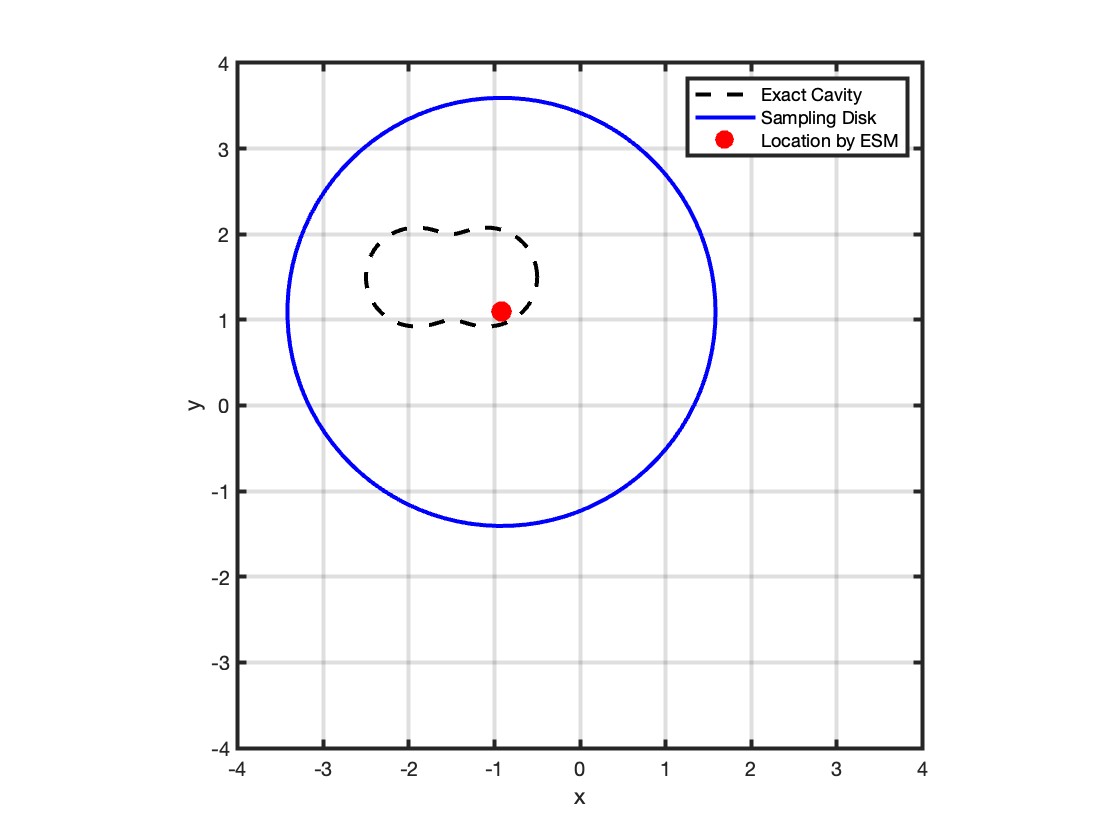}}
\caption{Reconstruction results using multilevel ESM for (a) the peanut-shaped cavity centered at the origin and (b) the peanut-shaped cavity shifted to $(-1.5,1.5)$, based on a single incident direction $d_0$.}\label{fig:esmpeanut}
\end{figure}
 
\begin{figure}[htp]
\centering	
\subfigure[The peach-shaped cavity centered at the origin]{\includegraphics[width=0.48\textwidth]{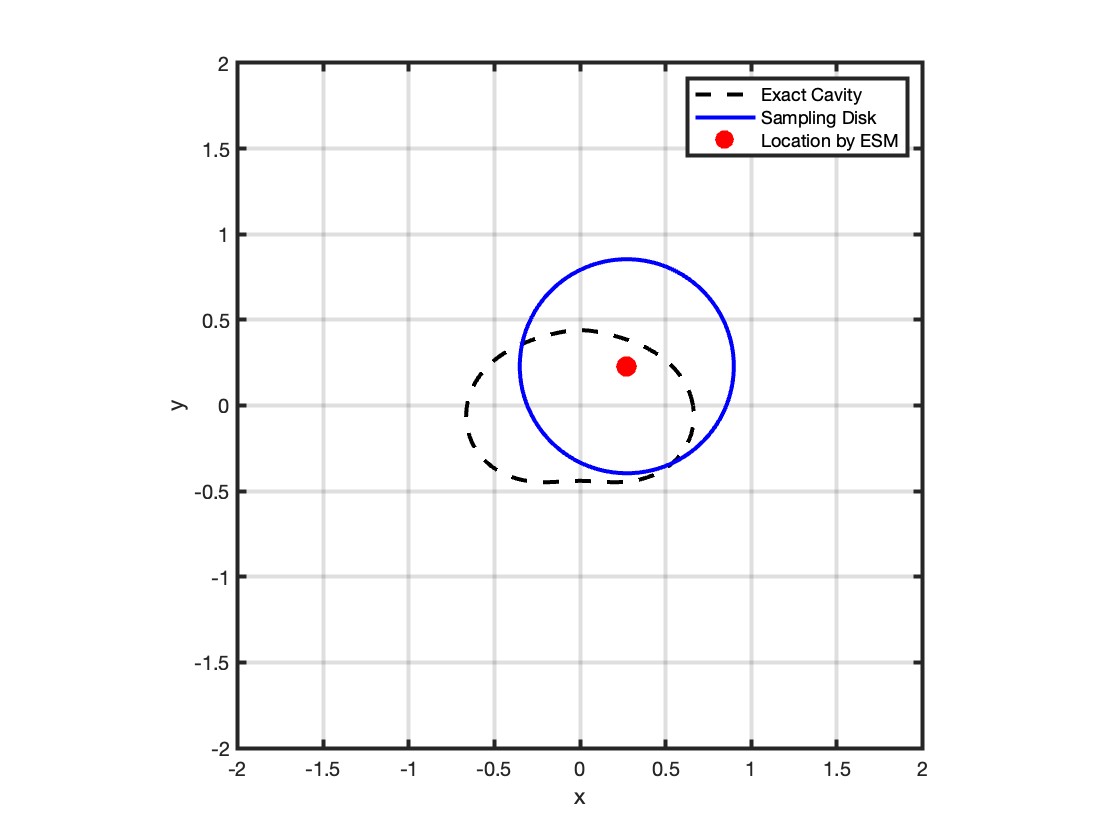}}\quad
\subfigure[The peach-shaped cavity shifted to $(-1.5,1.5)$]{\includegraphics[width=0.48\textwidth]{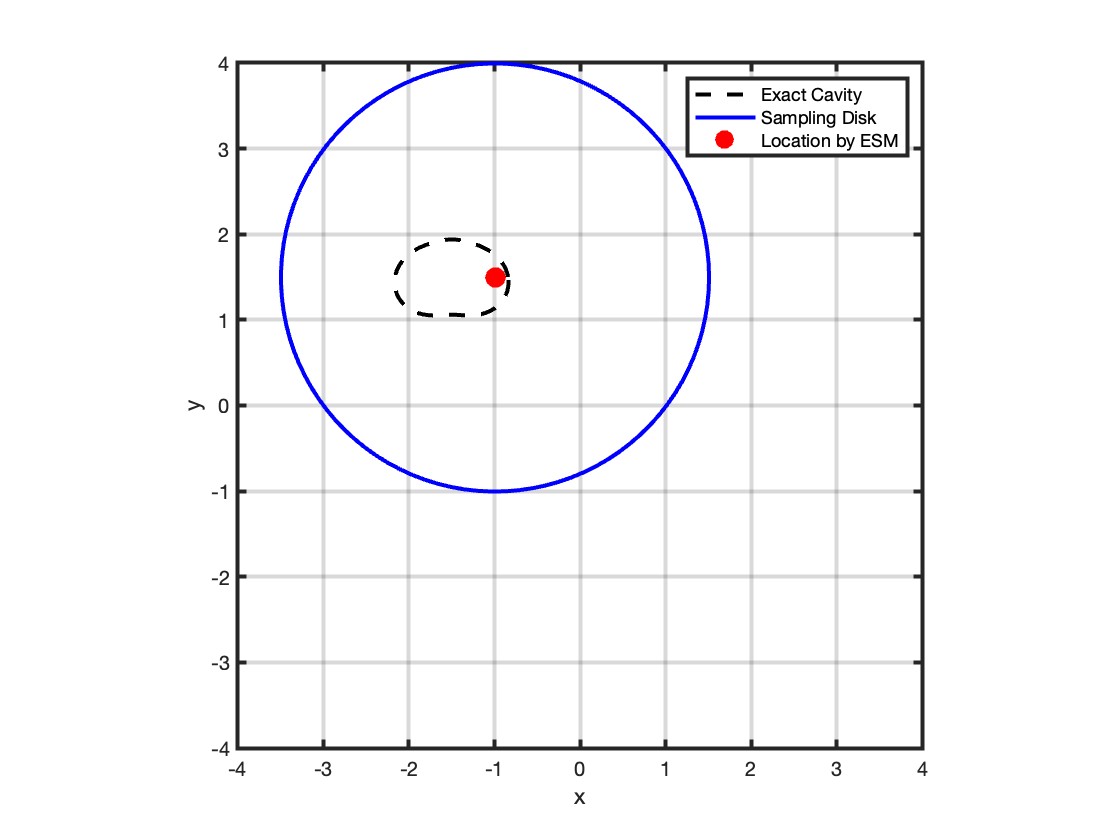}}
\caption{Reconstruction results using multilevel ESM for (a) the peach-shaped cavity centered at the origin and (b) the peach-shaped cavity shifted to $(-1.5,1.5)$, based on a single incident direction $d_0$.}\label{fig:esmpeach}
\end{figure}
 
\subsubsection{Multi-incident directions}

Selecting an appropriate radius $R$ for the sampling disks is critical for accurately reconstructing the location of clamped cavities when only a single incident direction is used. If $R$ is too large or too small, the reconstruction accuracy deteriorates. Although multilevel ESM helps estimate a suitable value of $R$, this choice becomes less sensitive when far-field data from multiple incident directions are available.

Figure \ref{fig:multiincpeach} shows the reconstruction of the approximate location of the peach-shaped cavity using a single incident direction $d_0$, as well as $5$ and $10$ incident directions corresponding to the incident apertures $\mathbb{S}^1_{\text{inc},1} $ and $\mathbb{S}^1_{\text{inc},2} $, respectively. These incident apertures are defined as
\begin{align*}
\mathbb{S}^1_{\text{inc},1} &=\big\{(\cos{\theta},\sin{\theta})\,|\,\theta=j\pi/8, j=0, 1, \dots, 4\big\},\\
\mathbb{S}^1_{\text{inc},2} &=\big\{(\cos{\theta},\sin{\theta})\,|\,\theta=j\pi/5, j=0, 1, \dots, 9\big\}.
\end{align*}  
We set the sampling disk radius to $R=1$, use \( I=40 \) observation directions, and apply Tikhonov regularization with parameter \( \alpha=10^{-4} \). The wavenumber is fixed at \( \kappa=2\pi \). As the number of incident directions increases, the accuracy of the reconstructed cavity location improves, even with a fixed and non-optimized radius \( R \).

\begin{figure}[htp]
\centering	
\subfigure[Single direction $d_0$]{\includegraphics[width=0.48\textwidth]{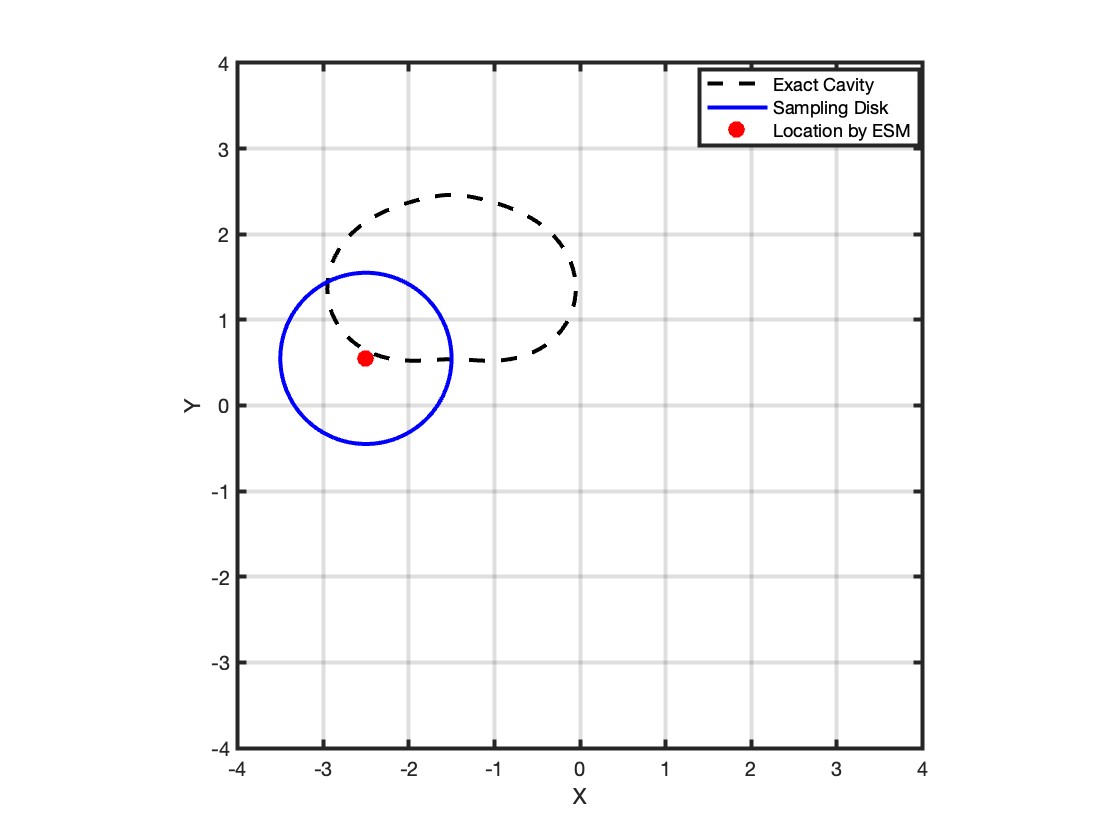}}\quad
\subfigure[Five directions $\mathbb{S}^1_{\text{inc},1} $]{\includegraphics[width=0.48\textwidth]{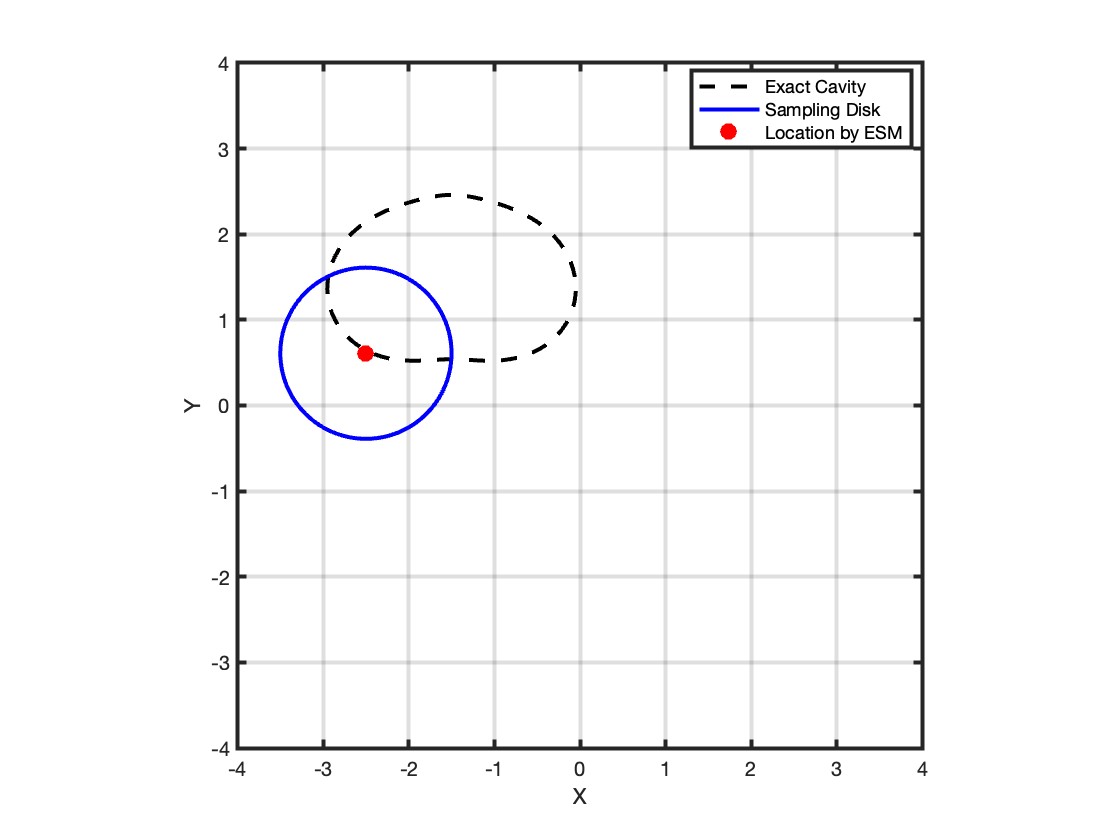}}\quad
\subfigure[Ten directions $\mathbb{S}^1_{\text{inc},2} $]{\includegraphics[width=0.48\textwidth]{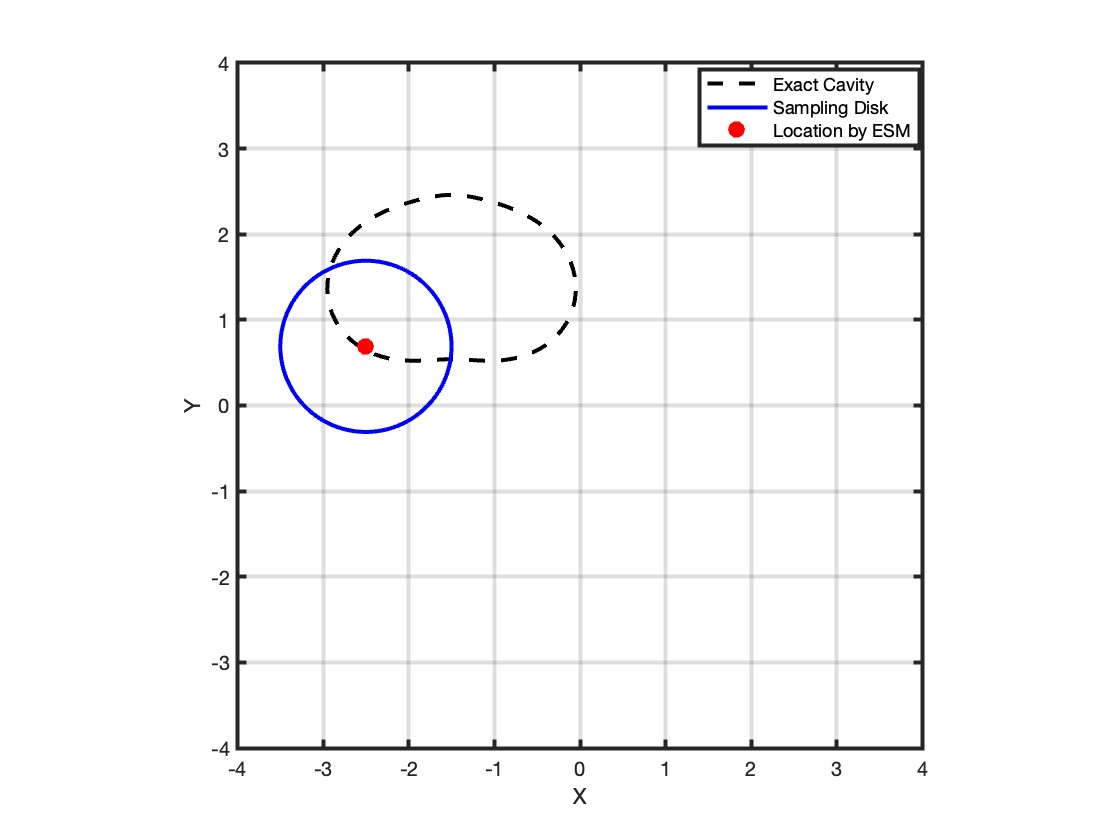}}
\caption{Reconstruction results of the ESM with multiple incident directions for the peach-shaped cavity centered at $(-1.5,1.5)$, using a fixed sampling disk radius $R=1$. Reconstructions are shown for the following incident apertures: (a) single direction $d_0$, (b) five directions $\mathbb{S}^1_{\text{inc},1} $, and (c) ten directions $\mathbb{S}^1_{\text{inc},2} $.}\label{fig:multiincpeach}
 \end{figure}
 
\subsubsection{Multi-frequency data}
 
 We present the implementation of the multiple-frequency ESM for the inverse biharmonic scattering problem with clamped boundary conditions. For the selected numerical experiments, we consider three frequency ranges: $[\kappa_{\min},\kappa_{\max}]=[\pi,2\pi],\, [\pi,4\pi]$, and $[\pi/3,5\pi]$. Each interval $[\kappa_{\min}, \kappa_{\max}]$ is uniformly divided into $L=5$ discrete wavenumbers given by
 \begin{align*}
     \kappa_{\ell}=\kappa_{\min}+(\ell+1)\frac{\kappa_{\max}-\kappa_{\min}}{L-1},\quad \ell=1,\dots,L.
 \end{align*}
As in previous examples, we compute the far-field data using the system of boundary integral equations:
 \begin{align*}
u^{\infty}( \hat{x} ,d,\kappa_{\ell}), \quad \text {for each} \quad \kappa_{\ell}\in [\kappa_{\min},\kappa_{\max}],\quad \ell=1,\dots,L,
 \end{align*}
 with a fixed incident direction $d$. We set $I=40$ observation directions for each wavenumber and use a uniform $200\times 200$ sampling grid. The Tikhonov regularization parameter is chosen as $\alpha=10^{-4}$. The discrete indicator function for multiple-frequency data at a fixed incident direction is given by
 \begin{align*}
\mathcal I(z)=\sum_{\ell=1}^L |\mathbf g_z^{\alpha}(d,\kappa_{\ell})|,\quad z\in\mathcal M.
 \end{align*}
 
 Figure \ref{fig:esmpeachfreq} shows the multi-frequency ESM reconstruction of the peach-shaped cavity using a fixed sampling radius of $R=1$. Similar results are observed for the apple- and peanut-shaped cavities. As with multiple incident directions, the advantage of using multi-frequency data is that the accuracy of the approximate location of the clamped cavity improves at a fixed radius $R$ as the frequency range, and hence the resolution, increases. The reconstructions remain accurate even with an arbitrarily chosen fixed radius, making the method less sensitive to the specific choice of $R$ when more frequency data are available. In contrast, the multilevel ESM iteratively selects an appropriate radius to improve the approximation of the cavity's location. However, if the radius is chosen too large or too small, the reconstruction quality may degrade significantly.

 \begin{figure}[htp]
\centering
\subfigure[Frequency range {$[\pi, 2\pi]$}]{\includegraphics[width=0.48\textwidth]{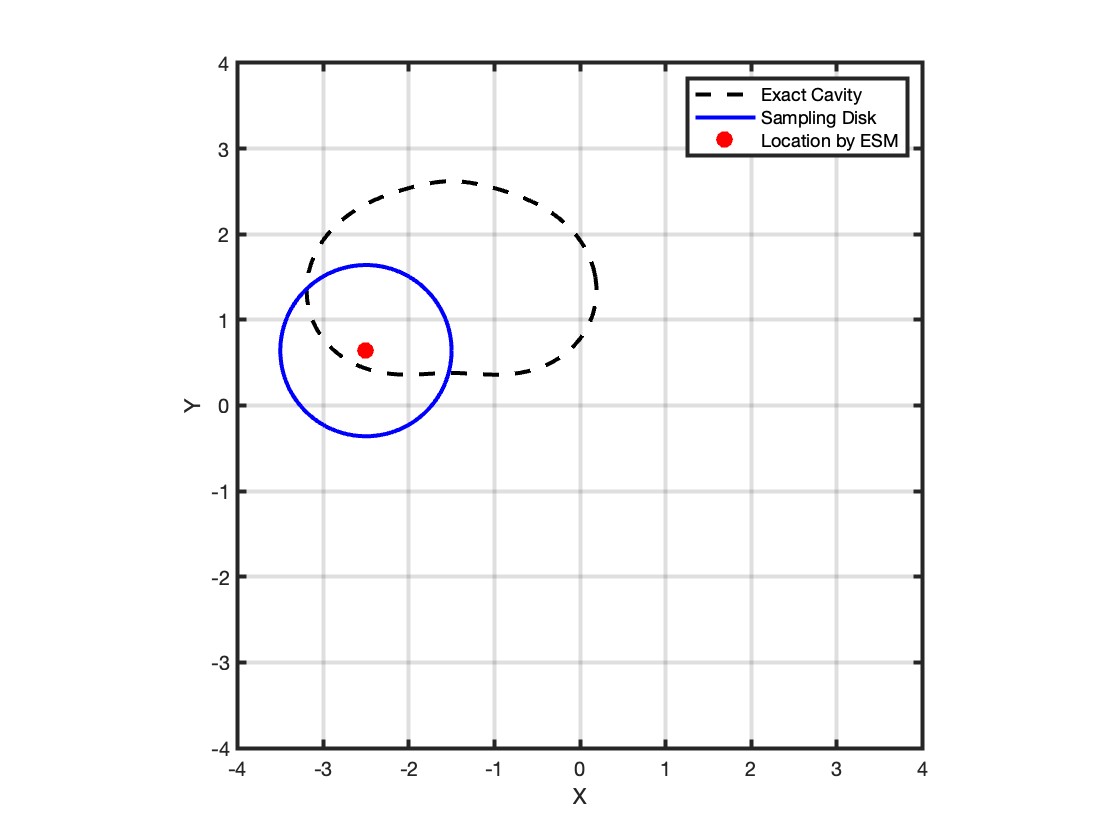}}\quad
\subfigure[Frequency range {$[\pi, 4\pi]$}]{\includegraphics[width=0.48\textwidth]{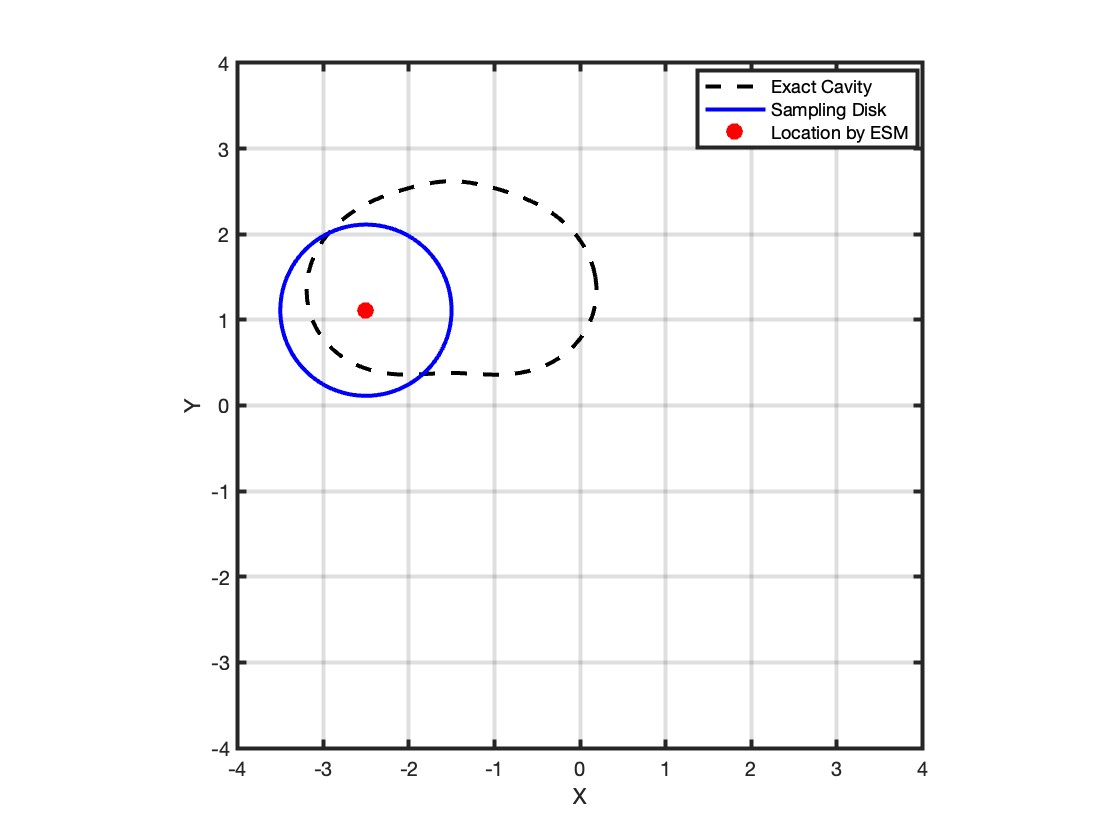}}\quad
\subfigure[Frequency range {$[\pi/3, 5\pi]$}]{\includegraphics[width=0.48\textwidth]{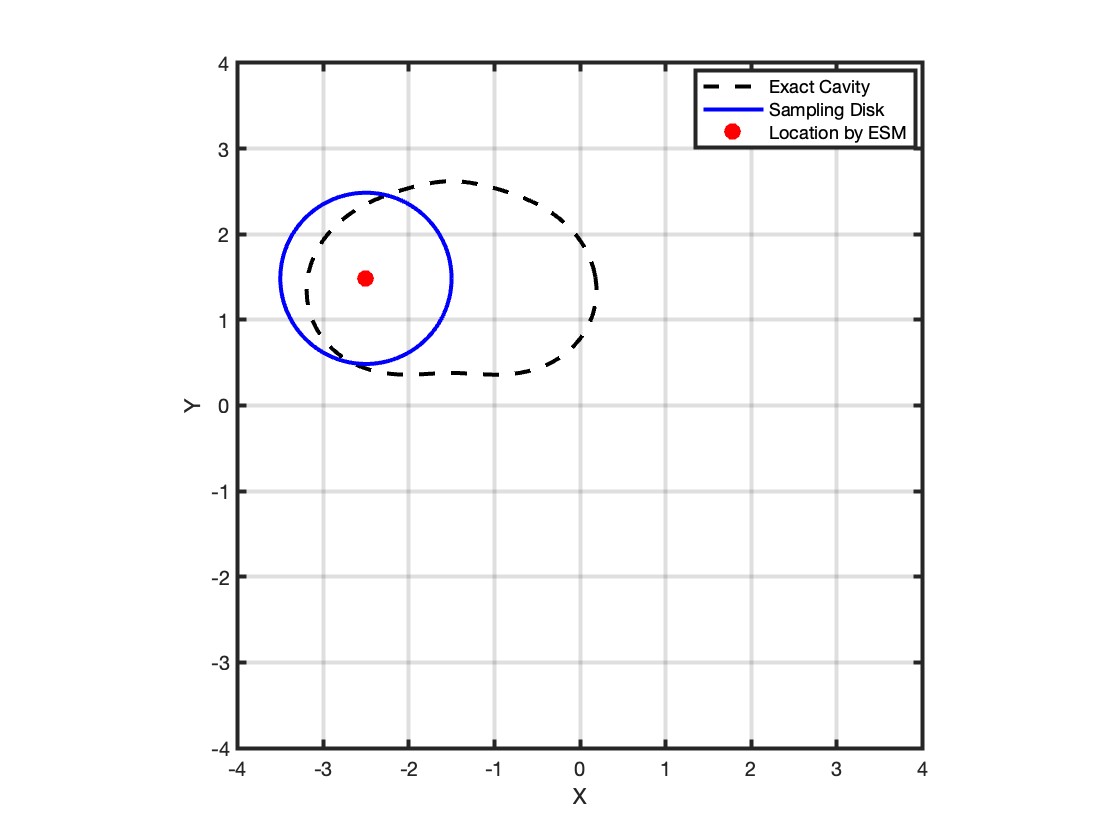}}
\caption{Reconstruction results of the multi-frequency ESM for the peach-shaped cavity shifted to $(-1.5,1.5)$ using a single incident direction $d_0$ and a fixed sampling radius $R=1$. Each frequency interval $[\kappa_{\min},\kappa_{\max}]$ is uniformly divided into $L=5$ wavenumbers. The reconstructions are shown for increasing frequency ranges: (a) $[\pi,2\pi]$, (b) $[\pi,4\pi]$, and (c) $[\pi/3,5\pi]$.}\label{fig:esmpeachfreq}
 \end{figure}
 
\section{Conclusion}\label{section 6}

In this paper, we have presented an alternative justification for the LSM based on far-field data, differing from \cite{guoetal2024} by requiring only the exclusion of eigenvalues associated with the clamped transmission problem. Notably, accurate reconstruction of clamped cavities remains possible even when the wavenumber corresponds to a Dirichlet eigenvalue of the negative Laplacian. The numerical experiments confirm the effectiveness of both the LSM and the ESM for the inverse cavity scattering problem of biharmonic waves in a Kirchhoff--Love plate, using far-field measurements. Furthermore, the indicator function exhibits robustness with respect to measurement noise, enabling reliable reconstruction of clamped cavities from Dirichlet boundary data.

Moreover, both multi-frequency ESM and ESM with multiple incident directions offer significant advantages by enhancing the accuracy of the approximate location of the clamped cavity, even when the sampling radius $R$ is fixed and arbitrary. As the frequency range or the number of incident directions increases, the reconstructions become more accurate. In contrast, when using a single incident direction at a fixed frequency, the choice of radius $R$ becomes more critical to ensure accurate reconstruction.
 
In comparison to the implementation of the LSM with near-field data in \cite{bourgeois2020linear}, the use of far-field data requires fewer measurements. When the observation points are sufficiently far from the cavity, the far-field pattern of the scattered field $u^s$ can be accurately approximated by the far-field pattern of its Helmholtz component $u^s_{\text{H}}$. This approximation reduces the amount of data needed for reliable reconstruction, making far-field methods more efficient than their near-field counterparts. 
 
Several open questions remain in the study of inverse biharmonic wave scattering. While this work focuses on the reconstruction of clamped cavities, future research may investigate the effectiveness of sampling methods in reconstructing cavities embedded in simply supported or free plates. Moreover, extending the LSM and ESM frameworks to accommodate penetrable cavities represents an interesting direction for further study.


\end{document}